\documentclass{amsart}
\usepackage{amsfonts,amssymb,amsmath,amsthm,amsbsy}
\usepackage{url}

\usepackage{enumerate}
\usepackage{stmaryrd}
\usepackage{tikz}
\usetikzlibrary{arrows,calc,matrix}
\usepackage{amsrefs}

\newtheorem{thrm}{Theorem}[section]
\newtheorem{lem}[thrm]{Lemma}
\newtheorem{prop}[thrm]{Proposition}

\theoremstyle{definition}
\newtheorem{definition}[thrm]{Definition}

\numberwithin{equation}{section}

\author{S. Madariaga}
\address{Departamento de Matem\'aticas y Computaci\'on, Universidad de
La Rioja, 26004 \\ Lo\-gro\-\~no, Spain}
\email{sara.madariaga@unirioja.es}

\author{Jos\'e M. P\'erez-Izquierdo}
\address{Departamento de Matem\'aticas y Computaci\'on, Universidad de
La Rioja, 26004 \\ Lo\-gro\-\~no, Spain}
\email{jm.perez@unirioja.es}

\thanks{The authors thank support from the Spanish Ministerio de Ciencia e Innovaci\'on (MTM2010-18370-C04-03).}%


\newcommand{\bm}{\boxminus}
\newcommand{\bp}{\boxplus}
\newcommand{\cB}{\mathbf{IBialg}}
\newcommand{\cC}{\mathbf{C}}

\newcommand{\cL}{\mathbf{Loops}}
\newcommand{\cM}{\mathbf{Moufang}}
\newcommand{\cS}{\mathbf{Sab}}
\newcommand{\cV}{\mathbf{V}}

\newcommand{\fd}{{f.d.}\ }
\newcommand{\field}{\boldsymbol{k}}

\newcommand{\m}{\mathfrak m}
\newcommand{\Oc}{\mathbb O}
\newcommand{\Prim}{\mathop{\mathrm{Prim}}}

\newcommand{\Z}{\mathbb Z}

\newcommand{\x}{\boldsymbol{x}}
\newcommand{\xx}{\boldsymbol{x}}
\newcommand{\yy}{\boldsymbol{y}}
\newcommand{\zz}{\boldsymbol{z}}
\newcommand{\llambda}{\boldsymbol{\lambda}}
\newcommand{\rrho}{\boldsymbol{\rho}}
\newcommand{\rr}{\boldsymbol{r}}
\newcommand{\sss}{\boldsymbol{s}}

\newcommand{\1}{_{(1)}}
\newcommand{\2}{_{(2)}}
\newcommand{\3}{_{(3)}}
\newcommand{\4}{_{(4)}}
\newcommand{\5}{_{(5)}}

\DeclareMathOperator{\ad}{\mathrm{ad}}

\DeclareMathOperator{\Atp}{\mathrm{Atp}}

\DeclareMathOperator{\Doro}{\mathcal{D}}
\DeclareMathOperator{\Endo}{\mathrm{End}}
\DeclareMathOperator{\Eq}{\mathrm{Eq}}
\DeclareMathOperator{\gl}{\mathfrak{gl}}

\DeclareMathOperator{\Id}{\mathrm{Id}}
\DeclareMathOperator{\Lie}{\mathcal{L}}

\DeclareMathOperator{\Mult}{\mathrm{Mlt}}
\DeclareMathOperator{\Nalt}{\mathrm{N}_{\mathrm{alt}}}
\DeclareMathOperator{\Na}{\mathrm{N}}

\DeclareMathOperator{\sl2}{\mathrm{sl}}
\DeclareMathOperator{\spann}{\mathrm{span}}
\DeclareMathOperator{\U}{\mathrm{U}}

\begin{document}

\title{Linear representations of formal loops}

\begin{abstract}
	A representation of an object in a category is an abelian group in the corresponding comma category. 
	In this paper we derive the formulas describing linear representations of objects in the category of formal loops 
	and apply them to obtain a new approach to the representation theory of formal Moufang loops and Malcev algebras 
	based on Moufang elements. Certain `non-associative Moufang symmetry' of groups is revealed.
\end{abstract}

\maketitle

%
%
%

\section{Introduction}

\emph{In this paper the base field $\field$ is assumed to be of characteristic zero.}

\medskip

Finite-dimensional real Lie algebras are the tangent spaces of real Lie groups. 
This crucial result was naturally extended to a non-associative setting using differential geometry \cite{Sabinin1988}. 
Loops play the role of non-associative groups and Sabinin algebras are the new `non-associative' Lie algebras. 
When convergence is not taken into account, local loops are replaced by formal loops in this non-associative Lie correspondence. 
The category of formal loops is equivalent to the category of irreducible unital non-associative bialgebras \cite{Mostovoy2010}. 
From an algebraic point of view this equivalence allows the substitution of bialgebras for loops, 
so instead of studying the Sabinin algebra of a formal loop 
one considers the Sabinin algebra of primitive elements of an irreducible bialgebra, 
as described by Shestakov and Umirbaev \cite{Shestakov2002}. 
The integration of a Sabinin algebra to a local analytic loop is replaced by 
the construction of the universal enveloping algebra of the Sabinin algebra, 
an irreducible unital non-associative bialgebra whose primitive elements form a Sabinin algebra 
which can be identified with the initial Sabinin algebra \cites{Perez-Izquierdo2007,Mostovoy2010}.

Lie algebras are essential to understand linear representations of local Lie groups, 
so it seems reasonable to expect that Sabinin algebras will be equally essential to understand linear representations 
of both local analytic loops and irreducible unital non-associative bialgebras. 
The main obstacle is that while there is a general consensus about what a linear representation of a Lie group is, 
the corresponding notion for loops is not so developed. 
For instance, different approaches to the representation theory of Moufang loops have been proposed \cites{Loginov1993,Dharwadker1995}. 
Sabinin already mentioned the lack of a natural representation theory for loops in his book \cite{Sabinin1999}.

In this paper we conciliate in an algebraic way the work on linear representations of loops \cites{Smith1986,Smith2007} 
with the work on the non-associative Lie correspondence \cites{Sabinin1988,Mostovoy2010}. 
Our goal is to study linear representations of formal loops by considering representations of the corresponding Sabinin algebra 
and conversely, to integrate representations of a Sabinin algebra to representations of its formal loop. 
We chose the approach of J. D. H. Smith to the representation theory of loops since it is based on 
an attractive categorical formulation by J. M. Beck \cite{Beck2003} that turns out to be very illuminating in our context.

We illustrate our techniques in the particular case of Moufang loops. 
Malcev algebras are the Sabinin algebras of local Moufang loops \cites{Maltsev1955,Kuzmin1972}, 
and the representation theory for them is well developed \cites{Kuzmin1968,Carlsson1976,Elduque1990,Elduque1995}. 
This theory is based on a standard definition of bimodule \cite{Eilenberg1948}.  
Malcev algebras include Lie algebras and other algebras such as the algebra of the traceless octonions. 
The theoretical existence of Malcev modules other than Lie modules for Lie algebras suggests certain kind of `Moufang symmetry' on groups. 
\footnote{We use the expression `Moufang symmetry' very loosely without defining it. E. Paal systematically approached this notion  \cites{Paal1990,Paal2008I,Paal2008II,Paal2008III,Paal2008IV,Paal2008V,Paal2008VI,Paal2008VII,Paal2008VIII,Paal2008X,Paal2008XI,Paal2008XII}. While the present work is modeled on the group $\Doro(Q)_e$, Paal's approach is modeled on the group $\Doro(Q)$ 
(see Section \ref{sec:Moufang_modules} for the definition of $\Doro(Q)$).}However, in practice, over algebraically closed fields of characteristic zero the only finite-dimensional simple Lie algebra 
that admits a non-Lie Malcev irreducible module is $\sl2(2,\field)$, and this module is unique and 2-dimensional \cites{Carlsson1976,Elduque1990}. Thus, no hope for that symmetry is left.

In this paper we propose a more general definition of a Malcev representation that we call \emph{relative representation}. 
This notion is the infinitesimal counterpart of the idea of relative representation for Moufang loops that we also introduce. 
We prove that any relative representation of a formal Moufang loop induces a relative representation of its Malcev algebra. 
Conversely, relative representations of Malcev algebras can be formally integrated to relative representations 
of the corresponding formal Moufang loops. 
Relative representations of Moufang loops are based on \emph{Moufang elements} \cite{Phillips2009}, 
i.e., elements $a$ in a loop such that
\begin{displaymath}
	a(x(ay)) = ((ax)a)y \quad \textrm{and} \quad ((xa)y)a = x(a(ya))
\end{displaymath}
for any $x,y$. 
By definition, any element in a Moufang loop is a Moufang element. 
However, Moufang elements might be present in non-Moufang loops and even in that case they always form a Moufang subloop. 
It seems natural to embed elements of Moufang loops as Moufang elements in other loops,
because they behave in the new loop as they do in the original loop. 
However, since groups are examples of Moufang loops, elements of a group might be embedded in other loops 
as Moufang elements without satisfying associativity with respect to the other elements in the loop. 
In this case the group would be exhibiting a Moufang symmetry rather than an associative symmetry. 
It is then apparent that a new approach to the representation theory of Moufang loops and Malcev algebras is required 
to understand this situation, specially since this Moufang symmetry occurs quite often. 
For instance, given a group $G$ and two linear representations $V$ and $W$ of $G$, the set $V \otimes W \times G$ with product
\begin{displaymath}
	\left(\sum v_i \otimes w_i,a\right) \hskip -2pt \left(\sum v'_j \otimes v'_j,b\right) \hskip -2pt 
	=
	\hskip -2pt  \left(\sum b^{-1}a^{-1}bav_i \otimes w_i + \sum b^{-1}a^{-2}b v'_j \otimes a w'_j,ab \right)
\end{displaymath}
is a loop and the subloop $0 \otimes 0 \times G$ isomorphic to $G$ consists of Moufang elements that, in general, 
do not associate with all the other elements.

This paper is structured as follows. 
In Section \ref{sec:Modules_for_loops} we review the representation theory for objects in the category of loops developed by J. D. H. Smith. 
In Section \ref{sec:Modules_Formal_Loops} we extend these results to objects in the category of irreducible unital non-associative bialgebras. 
Although this category is equivalent to the category of formal loops, in practice it is much more natural to work with bialgebras than with formal loops. 
In Section \ref{sec:Moufang_modules} we specialize these results to the case of formal Moufang loops and obtain a new notion of representation 
(\emph{relative representation}) for these loops and for Moufang-Hopf algebras. 
Representations of formal loops induce representations of the corresponding Sabinin algebras. 
In Section \ref{Modules_for_Malcev_algebras_revisited} we define \emph{relative representations} for Malcev algebras 
and classify these representations for central simple Malcev algebras. 
Finally, we prove the equivalence between relative representations of formal Moufang loops and Malcev algebras.

%
%
\section{Modules for loops} \label{sec:Modules_for_loops}
By a \emph{loop} we mean a non-empty set $Q$ endowed with a binary product $xy$ so that $(Q,xy)$ 
has a unit element denoted by $e$ and the left and right multiplication operators $L_x \colon y \mapsto xy$ and $R_x \colon y \mapsto yx$ 
are bijective for all $x \in Q$ \cite{Pflugfelder1990}. 
It is customary to consider two extra binary operations on $Q$, the so called \emph{left} and \emph{right divisions}:
\begin{displaymath}
	x\backslash y := L^{-1}_x(y) \quad \textrm{and} \quad y/x := R^{-1}_x(y).
\end{displaymath}
Clearly
\begin{displaymath}
	x\backslash (xy) = y =x( x \backslash y ), \quad  (yx)/x = y = (y/x)x \quad \textrm{and} \quad x \backslash x = y/y.
\end{displaymath}
These identities characterize loops.
\subsection{Modules for objects in a category with pullbacks}
The notion of module for a loop \cites{Smith1986,Smith2007} is an example of a general categorical definition given by Beck \cite{Beck2003}. 
Let $\cC$ be a category with pullbacks and $B$ an object in $\cC$. 
The \emph{comma category} $\cC \downarrow B$ \cite{MacLane1998} is the category whose objects are arrows $A \stackrel{\pi}{\rightarrow} B$ in $\cC$ 
and whose arrows are commutative diagrams

%
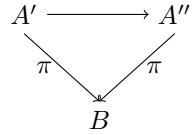
\begin{center}
 \begin{tikzpicture}
  \node(l) at (-1,0.7) {$A'$};
  \node(r) at (1,0.7) {$A''$};
  \node(d) at (0,-0.7) {$B$};

  \draw[->] ($(l.east)$) --  ($(r.west)$);
  \draw[->] ($(l.south)$) -- node[left]{$\pi$} ($(d.north)$);
  \draw[->] ($(r.south)$) -- node[right]{$\pi$} ($(d.north)$);
 \end{tikzpicture}
\end{center}
The identity arrow $B \rightarrow B$ is a terminal object in $\cC \downarrow B$. 
The existence of pullbacks in $\cC$ implies the existence of the product of any two objects 
$A' \stackrel{\pi}{\rightarrow} B$ and $A'' \stackrel{\pi}{\rightarrow} B$ in $\cC \downarrow B$, 
denoted by $A' \times_B A'' \stackrel{\pi}{\rightarrow} B$.

Groups are defined in categories with finite products and terminal objects \cites{Beck2003,MacLane1998}. 
A \emph{representation} of $B$ (or \emph{$B$-module}) is an \emph{abelian group} in $\cC \downarrow B$, 
i.e., an object $A \stackrel{\pi}{\rightarrow} B$ endowed with morphisms
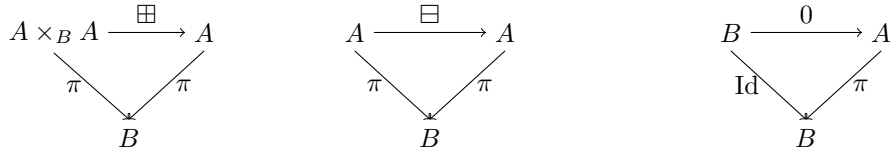
\begin{center}
	\begin{tikzpicture}
		\node(ll) at (-1,0.7) {$A\times_B A$};
		\node(rr) at (1,0.7) {$A$};
		\node(dd) at (0,-0.7) {$B$};
		\node(ll2) at (3,0.7) {$A$};
		\node(rr2) at (5,0.7) {$A$};
		\node(dd2) at (4,-0.7) {$B$};
		\node(ll3) at (8,0.7) {$B$};
		\node(rr3) at (10,0.7) {$A$};
		\node(dd3) at (9,-0.7) {$B$};
		
		\draw[->] ($(ll.east)$) -- node[above]{$\boxplus$} ($(rr.west)$);
		\draw[->] ($(ll.south)$) -- node[left]{$\pi$} ($(dd.north)$);
		\draw[->] ($(rr.south)$) -- node[right]{$\pi$}($(dd.north)$);
		\draw[->] ($(ll2.east)$) -- node[above]{$\boxminus$} ($(rr2.west)$);
		\draw[->] ($(ll2.south)$) -- node[left]{$\pi$} ($(dd2.north)$);
		\draw[->] ($(rr2.south)$) -- node[right]{$\pi$}($(dd2.north)$);
		\draw[->] ($(ll3.east)$) -- node[above]{$0$} ($(rr3.west)$);
		\draw[->] ($(ll3.south)$) -- node[left]{$\Id$} ($(dd3.north)$);
		\draw[->] ($(rr3.south)$) -- node[right]{$\pi$}($(dd3.north)$);
	\end{tikzpicture}
\end{center}
subject to the usual commutative diagrams for the \emph{addition map} $\bp$, the \emph{opposite map} $\bm$ 
and the \emph{zero map} $0$ of any abelian group \cite{Beck2003}.
\subsection{Modules for objects in $\cL$}
The objects of the category $\cL$ are loops and the arrows are homomorphisms of loops. 
The pullback of two arrows $Q_1 \stackrel{f_1}{\rightarrow} Q$ and $Q_2 \stackrel{f_2}{\rightarrow}Q$ is the loop 
$Q_1\times_Q Q_2 = \{(x_1,x_2) \in Q_1 \times Q_2 \mid f_1(x_1) = f_2(x_2)\}$ with arrows given by the projections onto $Q_1$ and $Q_2$. 
The existence of pullbacks in the category $\cL$ leads to a natural notion of representation of a loop. 
In case that our loop satisfies some identities, we can focus on the category given by the variety of loops determined by those identities 
instead of the entire category $\cL$. This leads to a more restrictive representation theory and 
working in $\cL$ provides the general framework to develop the representation theory of loops in other subcategories such as varieties. 
We briefly review the representation theory of objects in $\cL$. A very well-written exposition can be found in \cites{Smith1986,Smith2007}.

Given a module $(E\stackrel{\pi}{\rightarrow} Q, \bp, \bm, 0)$ for an object $Q$ in $\cL$,  the commutative diagram
\begin{center}
	\begin{tikzpicture}
		\node(ll) at (-1,0.7) {$Q$};
		\node(rr) at (1,0.7) {$E$};
		\node(dd) at (0,-0.7) {$Q$};
		
		\draw[->] ($(ll.east)$) -- node[above]{$0$} ($(rr.west)$);
		\draw[->] ($(ll.south)$) -- node[left]{$\Id$} ($(dd.north)$);
		\draw[->] ($(rr.south)$) -- node[right]{$\pi$}($(dd.north)$);
	\end{tikzpicture}
\end{center}
says that the exact sequence $\ker (\pi) \rightarrowtail E \stackrel{\pi}{\twoheadrightarrow} Q$ splits. 
The \emph{fiber} over $a \in Q$ is the set  $E_a := \{ x \in E \mid \pi(x) = a\}$. 
The fiber $E_e = \ker(\pi)$ is a normal subloop of $E$. 
The image of $a$ under $0$ will be denoted by $0_a$. 
We also denote elements from $E_a$ by $x_a, y_a,\dots$. 
The set $0_Q = \{ 0_a \mid a \in Q\}$ is a subloop of $E$ isomorphic to $Q$. 
Any $x \in E$ can be written as $x = (x/0_{\pi(x)})0_{\pi(x)}$ and $(x/0_{\pi(x)}) \in E_e$, thus the map
\begin{align*}
	E_e \times Q &\rightarrow  E\\
	(x_e,a) &\mapsto x_e 0_a
\end{align*}
is bijective. 
Since $\bp$ and $\bm$ preserve fibers, the commuting diagrams satisfied by $\bp, \bm$ and $0$ 
in the definition of abelian group in $\cL \downarrow Q$ imply that $E_a$ is an abelian group with respect to the addition map
\begin{displaymath}
	(x_a,y_a) \mapsto x_a + y_a :=\bp(x_a,y_a),
\end{displaymath}
the opposite $-x_a$ of $x_a$ being $\bm(x_a)$, and with zero element $0_a$. 
The map $\bp$ is a homomorphism of loops, hence $(x_e + y_e)0_a = \bp(x_e,y_e)\bp(0_a,0_a) = \bp(x_e0_a,y_e0_a) = x_e0_a + y_e0_a$. 
This shows that all the fibers $E_a$ are isomorphic as abelian groups. 
In fact, if we write $Q[\boldsymbol{x}]$ for the free product of $Q$ and the free loop on one generator $\boldsymbol{x}$, 
then the group $\U(Q;\cL)$ generated by the left and right multiplication operators $L_a,R_a \colon Q[\boldsymbol{x}] \rightarrow Q[\boldsymbol{x}]$ 
by elements $a \in Q$ acts on $E$ by $L_ax :=0_{a}x$ and $R_ax:=x0_{a}$, inducing isomorphisms between the fibers. 
This defines an action of the subgroup $\U(Q;\cL)_e := \{\phi \in \U(Q;\cL) \mid \phi (e) = e\}$ on the abelian group $E_e$, 
and so the abelian group $E_e$ is a $\U(Q;\cL)_e$-module. 
Under the identification of $E$ with $E_e \times Q$, the maps $\bp$, $\bm$ and $0$ correspond to
\begin{equation} \label{eq:plusminuszero}
	 ((x_e,a),(y_e,a)) \stackrel{\bp}{\mapsto} (x_e+y_e,a),\quad (x_e,a) \stackrel{\bm}{\mapsto} (-x_e,a) 
	 \quad\text{and}\quad  
	 a \stackrel{0}{\mapsto} (0_e,a)	
\end{equation}	
respectively. We can transport the loop structure of $E$ to $E_e \times Q$.  We have that
\begin{align*}
	(x_e0_a)(y_e0_b)  &= \bp(x_e0_a,0_a)\bp(0_b,y_e0_b) = \bp((x_e0_a)0_b,0_a(y_e0_b)) \\
	&= \bp((((x_e0_a)0_b)/0_{ab})0_{ab},((0_a(y_e0_b))/0_{ab})0_{ab})\\
	&= \bp(((x_e0_a)0_b)/0_{ab},(0_a(y_e0_b))/0_{ab})\bp(0_{ab},0_{ab})\\
	&= \left( ((x_e0_a)0_b)/0_{ab}+(0_a(y_e0_b))/0_{ab}\right) 0_{ab}.
\end{align*}
Therefore, if we consider the elements
\begin{equation}
\label{eq:rs}
	r(a,b) := R^{-1}_{ab}R_bR_a \quad \text{and} \quad s(a,b):=R^{-1}_{ab}L_aR_b
\end{equation}
in $\U(Q;\cL)_e$, the product on $E$ corresponds to
\begin{equation}
\label{eq:Smith}
(x_e,a)(y_e,b) :=(r(a,b)x_e + s(a,b)y_e, ab)
\end{equation}
on $E_e \times Q$. 
This process can be reversed: from any abelian group $E_e$ and any action of $\U(Q;\cL)_e$ regarded as an automorphism of $E_e$ 
one can obtain a $Q$-module $(E_e \times Q \stackrel{\pi}{\rightarrow}Q, \bp,\bm,0)$ by (\ref{eq:plusminuszero}) and (\ref{eq:Smith}), 
proving that $Q$-modules are equivalent to $\U(Q;\cL)_e$-modules.

\subsection{Modules for loops in a variety}
The representation theory for loops in a variety $\cV$ is an example of the representation theory of objects in $\cL$. 
As before, a $Q$-module $(E \stackrel{\pi}{\rightarrow} Q, \bp, \bm, 0)$ for a loop $Q$ in a variety $\cV$, 
determines an abelian group structure on the fiber $E_e$ over $e$ and an action of the group $\U(Q;\cL)_e$ by automorphisms on $E_e$ 
so we can recover $(E \stackrel{\pi}{\rightarrow} Q, \bp, \bm, 0)$ from the abelian group $E_e$ and the action of $\U(Q;\cL)_e$. 
In fact, a better choice than $\U(Q;\cL)_e$ is natural in this context. 
If $Q*\cV[\boldsymbol{x}]$ denotes 
the free product in the variety $\cV$ of the loop $Q$ and the free group $\cV[\boldsymbol{x}]$ on one generator $\boldsymbol{x}$, 
then the group $\U(Q;\cV)$ generated by the left and right multiplication operators 
$L_a, R_a \colon Q*\cV[\boldsymbol{x}] \rightarrow Q*\cV[\boldsymbol{x}]$ by elements $a \in Q$ acts on $E$ 
(note that $E$ is a loop in the variety $\cV$) 
and the subgroup $\U(Q;\cV)_e := \{ \phi \in \U(Q;\cV) \mid \phi(e)= e\}$ acts on $E_e$ as automorphisms. 
It is then natural to replace $\U(Q;\cL)_e$ by $\U(Q;\cV)_e$. 
However, when recovering a $Q$-module from a $\U(Q;\cV)_e$-module, 
we have to check that the loop $E_e \times Q$ given by (\ref{eq:Smith}) belongs to $\cV$. 
This might not happen because it is equivalent to the vanishing of the action of certain elements of the group algebra $\Z\U(Q;\cL)_e$. 
Therefore, the representation theory of a loop $Q$ in a variety $\cV$ is equivalent to 
the representation theory of a quotient of the group algebra $\Z\U(Q;\cV)_e$ instead of to 
the representation theory of $\Z\U(Q;\cV)_e$ itself. 
See Section 10.5 in \cite{Smith2007} for the details.
%
%
\section{Modules for formal loops} \label{sec:Modules_Formal_Loops}
The goal of this section is to specialize the notion of representation in the sense of Beck
to the category of formal loops following the work of Smith. 
Since it is more natural to work in the equivalent category of irreducible unital non-associative bialgebras, we will do so. 
The theory of formal loops needed for this approach has been developed in \cite{Mostovoy2010}.

Recall \cites{Sweedler1969,Abe1980} that a \emph{coalgebra} $(C,\Delta, \epsilon)$ is a vector space $C$ equipped with two linear operations 
$\Delta\colon C \rightarrow C \otimes C$ (\emph{comultiplication}) and $\epsilon \colon C \rightarrow \field$ (\emph{counit}) so that
\begin{displaymath}
	\sum \epsilon(u\1) u\2 = u = \sum \epsilon(u\2)u,
\end{displaymath}
where $\sum u\1 \otimes u\2$  is the usual Sweedler's notation for $\Delta(u)$. 
The coalgebra $C$ is called \emph{coassociative} in case that $\Delta \otimes \Id = \Id \otimes \Delta$
and it is called \emph{cocommutative} if $\sum u\1 \otimes u\2 = \sum u\2 \otimes u\1$ for all $u \in C$. 
Coassociativity implies that the element
\begin{displaymath}
	\sum u\1 \otimes u\2 \otimes \cdots \otimes u_{(n+1)} := \sum u\1 \otimes \cdots \otimes \Delta(u_{(i)}) \otimes \cdots \otimes u_{(n)}
\end{displaymath}
is well defined (it does not depend on the position $i$ where we apply $\Delta$) 
\footnote{All the coalgebras that we consider in this paper are coassociative and cocommutative 
so no special care about the subindices in these tensors is required.}. 
A \emph{coalgebra morphism} between the coalgebras $(C, \Delta, \epsilon)$ and $(C', \Delta', \epsilon')$ 
is a linear map $\theta \colon C \rightarrow C'$ that satisfies $(\theta \otimes \theta) \Delta = \Delta' \theta$ and $\epsilon'\theta = \epsilon$. 
A \emph{unital bialgebra} $(A, \Delta, \epsilon, \mu, \eta)$ is a coalgebra $(A, \Delta, \epsilon)$ with two extra linear maps 
$\mu \colon A \otimes A \rightarrow A$ (\emph{multiplication} or \emph{product}) and $\eta \colon \field \rightarrow A$ (\emph{unit}) so that
\begin{displaymath}
	\eta(\alpha) u = \alpha u = u \eta(\alpha) \quad \forall \alpha \in \field,\, u \in A,
\end{displaymath}
where $uv$ stands for $\mu(u \otimes v)$, and both $\mu$ and $\eta$ are coalgebra morphisms. 
Note that $A \otimes A$ is a coalgebra with structure maps $u \otimes v \mapsto \sum (u\1 \otimes v \1) \otimes (u\2 \otimes v\2)$ 
and $u \otimes v \mapsto \epsilon(u)\epsilon(v)$, and we can regard the base field $\field$ as a coalgebra 
with $\Delta \colon \alpha \mapsto \alpha 1 \otimes 1$ and $\epsilon \colon \alpha \mapsto \alpha$). 
The image of $1 \in \field$ by $\eta$ is denoted by $1$ and is the unit element of $A$. 
The adjectives cocommutative or coassociative apply to unital bialgebras in accordance with the properties of the underlying coalgebra. 
The paradigm of coalgebra in this paper is the symmetric algebra $\field[V]$ of a vector space $V$. 
This commutative algebra is a coalgebra $(\field[V],\Delta, \epsilon)$ (moreover a unital bialgebra) with the structure maps determined by
\begin{equation} \label{eq:symmetric}
	\Delta(a) := a \otimes 1 + 1 \otimes a \quad \textrm{and} \quad \epsilon(a) = 0
\end{equation}
for any $a \in V$ and extending them to homomorphisms of unital algebras $\Delta \colon \field[V] \rightarrow \field[V] \otimes \field[V]$ 
and $\epsilon \colon \field[V] \rightarrow \field$. 
The vector space $V$ is recovered in the coalgebra $\field[V]$ as the space of \emph{primitive elements}, i.e., 
those $a \in \field[V]$ such that $\Delta(a) = a \otimes 1 + 1 \otimes a$. 
The term \emph{irreducible unital bialgebra} is used in this paper to designate these unital bialgebras $(A,\Delta, \epsilon, \mu, \nu)$ 
whose underlying coalgebra $(A,\Delta,\epsilon)$ is isomorphic to a coalgebra $\field[V]$ where $V = \Prim A$. 
Since $\field 1$ is a simple subcoalgebra of these bialgebras, this is equivalent to $(A,\Delta,\epsilon)$ being irreducible as a coalgebra \cite{Sweedler1969}. 
Observe that we do not assume associativity in the definition of (unital) bialgebra. 
In fact, most of the bialgebras that appear in this paper are \emph{not necessarily associative} (also called \emph{non-associative}) 
and the reader should implicitly assume that.
\emph{Hopf algebras} $H$ are associative unital bialgebras with a linear map $S \colon H \rightarrow H$, the \emph{antipode}, 
satisfying $\sum S(u\1)u\2 = \epsilon(u) = \sum u\1 S(u\2)$. 
Since $\sum S(u\1)(u\2 v) = \epsilon(u)v = \sum (vu\1)S(u\2)$, the antipode ensures some sort of cancellative property in $H$. 
This is no longer true for general bialgebras, but irreducible unital bialgebras are rather friendly \cite{Perez-Izquierdo2007}. 
Any such bialgebra $A$ always have two extra bilinear maps (coalgebra morphisms) $\backslash$ (\emph{left division}) 
and $/$ (\emph{right division}) such that
\begin{displaymath}
	\sum u\1 \backslash (u\2 v) = \epsilon(u)v = \sum u\1(u\2 \backslash v)
\end{displaymath}
and
\begin{displaymath}
	\sum (vu\1)/u\2 = \epsilon(u)v = \sum (v/u\1)u\2.
\end{displaymath}
This can be proved by induction using the \emph{coradical filtration}, i.e., the usual filtration by degree on the symmetric algebra $\field[\Prim A]$ 
when we look at $A$ as $\field[\Prim A]$.  
The elements $u\backslash v$ and $v/u$ can be obtained as linear combinations of iterated products of $v$ and elements taken from $\{ u\1, u\2,u\3,\dots\}$ 
so any sub-bialgebra inherits these divisions too. 
In some cases, such as Hopf algebras or in general Hopf quasigroups \cite{Klim2010}, $u\backslash v = S(u)v$ and $u/v = uS(v)$ for some map $S$ 
that could be rightfully called \emph{antipode}, but that is not always the case for arbitrary unital irreducible bialgebras.
%
\subsection{Formal loops}
%
Let $V$ be a $\field$-vector space and $\field[V]$ be the symmetric algebra of $V$. 
There exits a canonical isomorphism between $\field[V]\otimes \field[W]$ and $\field[V \times W]$ 
so that $V \otimes 1 + 1 \otimes W$ is identified with $V \times W$. 
We will consider the usual coalgebra structure (\ref{eq:symmetric}) on $\field[V]$, the space of primitive elements $\Prim \field[V]$ being $V$. 
Coalgebra morphisms $\field[V] \rightarrow \field[W]$ are encoded in formal maps. 
A \emph{formal map} from a $\field$-vector space $V$ to a $\field$-vector space $W$ is a linear map $\theta \colon \field[V] \rightarrow W$
with $\theta(1) = 0$. Any formal map $\theta \colon \field[V] \rightarrow W$ induces a unique coalgebra morphism 
$\theta'\colon \field[V] \rightarrow \field[W]$ with $\pi_W \theta' = \theta$, where $\pi_W$ denotes the projection of $\field[W]$ onto $W$. 
The precise formula for $\theta'$ is
\begin{displaymath}
	\theta'(u) = \sum_{n=0}^\infty \frac{1}{n!} \, \theta(u_{(1)})\cdots \theta(u_{(n)}) = \epsilon(u)1 + \theta(u) + \cdots,
\end{displaymath}
where the product is the usual associative multiplication on $\field[V]$. 
A \emph{formal loop} on $V$ is a formal map $F \colon \field[V \times V] \rightarrow V$
with $F\vert_{\field[V]\otimes 1} = \pi_V = F\vert_{1\otimes \field[V]}$. 
Therefore, a formal loop is equivalent to a unital bialgebra product $F' \colon \field[V]\otimes \field[V] \rightarrow \field[V]$.
The formal loop is recovered as $F = \pi_V F'$. A \emph{homomorphism} of formal loops $F\colon \field[V \times V] \rightarrow V$ 
and $H \colon \field[W \times W] \rightarrow W$ is a formal map $\theta \colon \field[V] \rightarrow W$ that verifies
\begin{displaymath}
	H'(\theta'(u) \otimes \theta'(v)) = \theta' (F'(u \otimes v))
\end{displaymath}
for any $u,v \in \field[V]$. 
Homomorphisms between formal loops correspond to homomorphisms between bialgebras. 
Moreover, the correspondence $F \mapsto F'$ and $\theta \mapsto \theta'$ gives
\begin{prop}\cite[Proposition 2.6]{Mostovoy2010}
	The category of formal loops and the category of irreducible unital bialgebras are equivalent.
\end{prop}
Under this equivalence the left and right divisions on a bialgebra correspond to the left and right divisions on the formal loop 
(see \cite{Mostovoy2010}), so these divisions are natural operations for bialgebras.

Identities in formal loops and bialgebras were considered in \cite{Mostovoy2010}. 
An identity of a loop $Q$ is an equality of two maps from $Q \times \cdots \times Q$ to $Q$ expressible in terms of the structure maps of $Q$. 
Note the unit element $e$ can be regarded as a $0$-ary operation and the multiplication and the left and right divisions as binary operations. 
For formal loops, identities are related to the equality of formal maps from $\field[V \times \cdots \times V]$ to $V$, 
while for coalgebras they are related to the equality of coalgebra morphisms from $\field[V] \otimes \cdots \otimes \field[V]$ to $\field[V]$. 
The following notation \cite{Mostovoy2010} works similarly for identities in loops and formal loops. 
The projection $\pi_{V_i} \colon \field[V_i] \rightarrow V_i$ is denoted by $\x_i$, 
the zero map $\field[V_i] \rightarrow V_i$ is denoted by $\boldsymbol{e}$ or $ \boldsymbol{0}$ and, 
given formal maps $\theta_i \colon \field[U_i] \rightarrow V_i$ $1 \leq i \leq n$ and $G \colon \field[V_1 \times \cdots \times V_n] \rightarrow W$, 
$G(\theta_1,\dots, \theta_n)$ stands for $G \circ \theta'_1 \otimes \cdots \theta'_n$. 
Although the maps $G$ and $G(\x_1,\dots, \x_n)$ are the same ($\boldsymbol{x}'_i$ is the identity map on $\field[V_i]$), 
the latter expression is much more in accordance with the standard notation in loop theory. 
If $F \colon \field[V \times V] \rightarrow V$ is a formal loop, we write $\boldsymbol{x}\boldsymbol{y}$ instead of $F$. 
To push further this connection between identities in loops and formal loops, 
we need to define what we mean by $G(\x_1,\dots, \x_i,\dots, \x_i,\dots, \x_n)$, i.e. when we allow repeated occurrences of $\x_i$. 
The simplest case is when $V_1 = \cdots = V_n$. In this case the notation $G(\x,\dots, \x)$ stands for the formal map 
$k[V] \rightarrow V$ given by $G(u) = \sum G(u\1 \otimes \cdots \otimes u_{(n)})$. 
One defines similarly $G(\x_{i_1} ,\dots,  \x_{i_n})$ when there are various groups of repeated indices among the $i_k$. 
For example, with this notation, the left and right Moufang identities for formal loops are
\begin{displaymath}
	\boldsymbol{x}(\boldsymbol{y}(\boldsymbol{x}\boldsymbol{z})) = ((\boldsymbol{x}\boldsymbol{y})\boldsymbol{x})\boldsymbol{z} 
	\quad \textrm{and}\quad 
	((\boldsymbol{z}\boldsymbol{x})\boldsymbol{y})\boldsymbol{x} = \boldsymbol{z}(\boldsymbol{x}(\boldsymbol{y}\boldsymbol{x})),
\end{displaymath}
which look like the usual Moufang identities, although they are artificial. 
These expressions are just a convenient way of representing the following equalities of coalgebra morphisms 
$\field[V] \otimes \field[V] \otimes \field[V] \rightarrow \field[V]$ (or their projection onto $V$):
\begin{displaymath}
\sum z\1 (u (z \2 v)) = \sum ((z\1u)z\2)v \quad \textrm{and} \quad \sum ((vz\1)u)z\2 = \sum v(z\1(u z\2))
\end{displaymath}
for all $u,v,z \in \field[V]$. 
The reader should keep in mind that with this notation identities on loops translate verbatim to formal loops, 
but that they represent \emph{multilinear} identities on bialgebras, and the use of the comultiplication and the counit is mandatory: 
just duplicating or removing elements from the identities induces the loss of the multilinearity. 
This approach to identities on bialgebras is motivated by the interpretation of these algebraic structures 
as distributions with support at a point of local analytic loops \cites{Perez-Izquierdo2007,Mostovoy2010}.
\subsection{The category of irreducible unital bialgebras}
Let $\cB$ be the category of irreducible unital $\field$-bialgebras. This category is known to be equivalent to the category $\cS$ of Sabinin algebras \cite{Mostovoy2010}, certain variety of algebras in the sense of universal algebra. Thus, many properties of $\cB$ such that the existence of finite products, terminal objects, equalizers and zero morphisms  are inherited from this equivalence. However, the functor from $\cS$ to $\cB$, i.e., the construction of universal enveloping algebras for Sabinin algebras, is far from being trivial \cite{Perez-Izquierdo2007}. So, for the convenience of the reader, we will provide a brief description of all categorical objects needed in $\cB$.

The \emph{product} of two objects $A_1$ and $A_2$ in $\cB$ is the tensor product $A_1 \otimes A_2$ with the projections
\begin{displaymath}
	\pi_1 \colon u \otimes v \mapsto \epsilon(v) u  \quad \text{and}\quad \pi_2 \colon u \otimes v \mapsto \epsilon(u) v.
\end{displaymath}
The \emph{equalizer} $\Eq(f,g)$ of two parallel arrows $A \stackrel{f}{\rightarrow} B$ and $A \stackrel{g}{\rightarrow} B$ in $\cB$ is the arrow $\Eq(f,g) \stackrel{eq}{\rightarrow} A$ where
\begin{displaymath}
	\Eq(f,g) := \sum_{\substack{C \text{ sub-coalgebra of } A \\ f\vert_C = g\vert_C}} C  \quad \text{and} \quad eq := f\vert_{\Eq(f,g)} = g \vert_{\Eq(f,g)}.
\end{displaymath}
The existence of finite products and equalizers ensures the existence of pullbacks in $\cB$. The \emph{pullback} of two arrows $A_1 \stackrel{\pi_1}{\rightarrow} B$ and $A_2 \stackrel{\pi_2}{\rightarrow} B$ consists of the bialgebra
\begin{displaymath}
	A_1 \otimes_B A_2 := \Eq(\pi_1\otimes \epsilon 1, \epsilon 1 \otimes \pi_2),
\end{displaymath}
i.e., the sum of all subcoalgebras $C$ of $A \otimes B$ such that $\pi_1\otimes \epsilon 1\vert_C = \epsilon 1 \otimes \pi_2\vert_C$, and the arrows $A_1 \otimes_B A_2 \stackrel{p_1}{\rightarrow} A_1$ and $A_1 \otimes_B A_2 \stackrel{p_2}{\rightarrow} A_2$ given by $\sum u_i \otimes v_i \mapsto \sum \epsilon(v_i)u_i$ and $\sum u_i \otimes v_i \mapsto \sum \epsilon(u_i) v_i$ respectively. The arrow
\begin{displaymath}
	A_1 \otimes_B A_2 \stackrel{\pi}{\rightarrow} B \quad \text{with} \quad \pi := \pi_1 p_1 = \pi_2 p_2
\end{displaymath}
is an object in the category $\cB \downarrow B$. In fact, this object, together with $p_1$ and $p_2$, is the product of the objects $A_1 \stackrel{\pi_1}{\rightarrow} B$ and $A_2 \stackrel{\pi_2}{\rightarrow} B$ in $\cB \downarrow B$. Since $B \stackrel{\Id}{\rightarrow} B$ is a \emph{terminal object} in $\cB \downarrow B$, the  usual properties of products in categories \cite{MacLane1998} show that
\begin{prop}
	Given arrows $A_i \rightarrow B$ ($i=1,2,3$) and $B \stackrel{\Id}{\rightarrow} B$ in $\cB$, we have that in $\cB \downarrow B$
	\begin{enumerate}
		\item $A_1 \otimes_B A_2 \cong A_2 \otimes_B A_1$,
		\item $(A_1 \otimes_B A_2) \otimes_B A_3 \cong A_1 \otimes_B(A_2 \otimes_B A_3)$ and
		\item $A_1 \otimes_B B\cong A_1 \cong B \otimes_B A_1$.
	\end{enumerate}
\end{prop}
The map $x\mapsto \epsilon(x)1$ is the \emph{zero morphism} $0_{AB}$ from $A$ to $B$ in $\cB$, hence this category has zero morphisms and equalizers. The \emph{kernel} of any morphism $A \stackrel{f}{\rightarrow} B$ is defined as $\ker(f) := \Eq(f,0_{AB})$, i.e.
\begin{displaymath}
	\ker(f) = \sum_{\substack{C \text{ is a sub-coalgebra of } A \\ f\vert_C = \epsilon 1\vert_C}} C
\end{displaymath}

\subsection{Modules for objects in $\cB$}
\label{subsec:ModulesForBialgebras}
Any abelian group $(E\stackrel{\pi}{\rightarrow} Q, \bp, \bm, 0)$ in $\cL \downarrow Q$ can be recovered from a structure of abelian group on the fiber $E_e = \ker (\pi)$ over $e$ as $E_e \times Q$ with structure maps given by (\ref{eq:plusminuszero}) and (\ref{eq:Smith}). We will extend this result to abelian groups in $\cB \downarrow B$ where $B$ is a fixed object in $\cB$. The kernel in $\cB$ of the arrow $A \stackrel{\pi}{\rightarrow} B$ will be denoted by $F(A)$. Since $F(A)$ is a sub-bialgebra of $A$, then it is also closed under the left and right division on $A$, i.e.
\begin{displaymath}
	F(A)F(A) + F(A)/F(A)+ F(A)\backslash F(A) \subseteq F(A).
\end{displaymath}
Also notice that for all $x \in F(A)$
\begin{displaymath}
\pi(x) = \epsilon(x) 1.
 \end{displaymath}
\subsubsection{Sections}
Although in general $(A \stackrel{\pi}{\rightarrow} B, \bp, \bm, 0)$ will denote an abelian group in $\cB \downarrow B$, however our initial results do not require that much structure but just the existence of a \emph{section} (of $\pi$), i.e., a homomorphism of bialgebras $B \stackrel{0}{\rightarrow} A$ such that $\pi \circ 0 = \Id_B$. The image of $b \in B$ by $B \stackrel{0}{\rightarrow} A$ can be denoted by $0_b$ or $0(b)$ but to avoid the awkward use of $0$ in our formulas we will identify $0(b)$ with $b$, so we will omit the map $0$ and we will freely write $xb \in A$ for $x \in A$ and $b \in B$. In the same way we will write $\pi(b) = b$ instead of $\pi0(b) = b$.
%
\begin{prop}
	We have that
	\begin{displaymath}
		F(A) = \left\lbrace \sum u_{(1)}/\pi(u_{(2)}) \mid u \in A \right\rbrace = \left\lbrace \sum \pi(u_{(1)}) \backslash u_{(2)} \mid u \in A \right\rbrace.
	\end{displaymath}
\end{prop}
\begin{proof}
	For any element $x \in F(A)$, $\pi(x) = \epsilon(x)1$. Hence, since $F(A)$ is a sub-coalgebra, $$\sum x\1 /\pi(x\2) = \sum \epsilon(x\2) x\1/1 = x/1 = x.$$ This proves that $F(A) \subseteq C$ with $C:=\left\lbrace \sum u_{(1)}/\pi(u_{(2)}) \mid u \in A \right\rbrace$. Conversely, for any $u \in A$ we have that
	\begin{displaymath}
		\Delta\left(\sum u\1 / \pi(u\2)\right) = \sum u\1 / \pi(u\2) \otimes u\3 / \pi(u\4)
	\end{displaymath}
	so $C$ is a sub-coalgebra of $A$. Moreover,
	\begin{align*}
		\pi\left(\sum u\1 /\pi(u\2)\right) &= \sum \pi(u\1) / \pi(u\2) \\
		&= \epsilon(u)1 = \epsilon\left(\sum u\1 /\pi(u\2)\right)1.
	\end{align*}
	implies that $C \subseteq F(A)$. The other equality can be proved in a similar way.
\end{proof}

\begin{prop}
\label{prop:factorization}
	The maps
	\begin{displaymath}
		\begin{array}{c@{\ }l@{\ }c}
			F(A) \otimes B & \rightarrow & A \\
			x \otimes b & \mapsto & xb
		\end{array}
		\quad\text{and}\quad
		\begin{array}{c@{\ }l@{\ }c}
					B \otimes F(A) & \rightarrow & A \\
					b \otimes x & \mapsto & bx
				\end{array}
		\end{displaymath}
	are isomorphisms of coalgebras.
\end{prop}
\begin{proof}
	We will only prove that the first map is a linear isomorphism since it obviously is a morphism of coalgebras. Any $u \in A$ can be written as
	\begin{displaymath}
		u = \sum (u\1/\pi(u\2))\pi(u\3) \in F(A)B
	\end{displaymath}
	so $A = F(A)B$. This proves the surjectivity. To prove the injectivity, assume that $\sum x_i b_i = 0$ for some $\sum x_i \otimes b_i \in F(A) \otimes B$. Then, $\sum_i {x_i}\1 {b_i}\1 \otimes {x_i}\2 {b_i}\2 = 0 \otimes 0$, so applying $\Id \otimes \pi$ we get
	\begin{displaymath}
		\sum_i x_i {b_i}\1 \otimes {b_i}\2 = 0 \otimes 0.
	\end{displaymath}
	Using the comultiplication $\Delta$ we obtain that
	\begin{displaymath}
		\sum_i x_i {b_i}\1 \otimes {b_i}\2 \otimes {b_i}\3 = 0 \otimes 0 \otimes 0.
	\end{displaymath}
	Finally, we apply $/ \otimes \Id$ to get
	\begin{displaymath}
		\sum_i x_i \otimes b_i = 0 \otimes 0.
	\end{displaymath}
\end{proof}
\begin{prop}
\label{prop:F_AxA}
	Let $A_i \stackrel{\pi}{\rightarrow} B$ $i = 1, 2$ be objects in $\cB \downarrow B$ and $B \stackrel{0}{\rightarrow} A_i$ be sections $i=1,2$. Then
	\begin{enumerate}
	\item $0 \colon b \mapsto \sum b\1 \otimes b\2$ is a section of $A_1 \otimes_B A_2 \stackrel{\pi}{\rightarrow} B$,
	\item $F(A_1 \otimes_B A_2) = F(A_1) \otimes F(A_2)$ and
	\item  the following map is an isomorphism of coalgebras
			\begin{align*}
				F(A_1) \otimes F(A_2) \otimes B & \rightarrow  A_1 \otimes_B A_2\\
				x \otimes y \otimes b & \mapsto  \sum x  b\1 \otimes y b\2.
			\end{align*}
	\end{enumerate}
\end{prop}
\begin{proof}
	(1) is obvious and (3) is a consequence of Proposition \ref{prop:factorization} and (2).
	To show (2), note that $F(A_1) \otimes F(A_2)$ is a subcoalgebra of $A_1 \otimes A_2$ on which 
	the restriction of $\epsilon1 \otimes \pi_2$,  $\pi_1 \otimes \epsilon 1$ and $\epsilon 1 \otimes \epsilon 1$ agree, 
	so we have $F(A_1) \otimes F(A_2) \subseteq F(A_1 \otimes_B A_2)$. 
	Conversely, for any $\sum x_i \otimes y_i \in F(A_1 \otimes_B A_2)$
	we have $\sum \pi(x_i) \otimes  \pi(y_i) = \sum \epsilon(x_i) \epsilon(y_i) 1 \otimes 1$. 
	Thus, 
	$\sum x_i \otimes y_i = \sum {x_i}\1/\epsilon({x_i}\2) \otimes {y_i}\1 / \epsilon({y_i}\2) 
	= 
	\sum {x_i}\1/\pi_1({x_i}\2) \otimes {y_i}\1 / \pi_2({y_i}\2) \in F(A_1) \otimes F(A_2)$.
\end{proof}

In the rest of this section we assume that $(A \stackrel{\pi}{\rightarrow} B, \bp,\bm,0)$ is an abelian group in $\cB \downarrow B$. 
In particular,
\begin{displaymath}
	\bp\left( \sum {x_b}\1 \otimes b\2 \right) = \bp \left( \sum {x_b}\1 \otimes 0_{b\2} \right)=x_b = \bp \left( \sum b\1 \otimes {x_b}\2 \right)
\end{displaymath}
for any $x_b$ with $\pi(x_b)=b$. 
In the case that $b = 1$ we get $\bp(x \otimes 1) = x = \bp(1 \otimes x)$ for any $x \in F(A)$. In case that $x = 0_b$ we obtain that
\begin{equation}
\label{eq:zero_element}
	\bp\left( \sum b\1 \otimes b\2 \right) = b.
\end{equation}

\subsubsection{Modules for objects in $\cB$}
The following results lead to the description of the product of $A$ in terms of the products of $F(A)$ and $B$. 
\begin{prop}
	For any $x, x' \in F(A)$ we have that
	\begin{displaymath}
		xx' = x \bp x'.
	\end{displaymath}
	So $F(A)$ is an associative and commutative sub-bialgebra of $A$ isomorphic to the symmetric algebra of $\Prim F(A)$.
\end{prop}
\begin{proof}
	Since $\bp$ is a homomorphism of bialgebras,
	\begin{align*}
		\bp(x \otimes x') &= \bp(x 1 \otimes 1x') = \bp((x \otimes 1)(1 \otimes x')) \\
		&= \bp(x \otimes 1) \bp(1 \otimes x') = xx'.
	\end{align*}
	Associativity and commutativity follow from the axioms of abelian group satisfied by $(A\stackrel{\pi}{\rightarrow} B, \bp, \bm, 0)$. 
	The Cartier-Milnor-Moore theorem \cite[Theorem 13.0.1]{Sweedler1969} implies that $F(A)$ is isomorphic to the symmetric algebra of $\Prim F(A)$.
\end{proof}
In order to avoid the occurrence of confusing parentheses in our formulas, we will use $\cdot$ as a separator to denote the product. 
For example, $xb \cdot x'b'$ represents the element $(xb)(x'b')$. 
The statement of Proposition \ref{prop:product} illustrates the convenience of this notation.
\begin{prop} \label{prop:distributive}
	For any $x,x' \in F(A)$ and $b \in B$ we have that
	\begin{displaymath}
		(xx')b = \bp\left(\sum  xb\1 \otimes x'b\2\right) \quad \text{and} \quad b(xx') = \bp\left(\sum b\1 x \otimes b\2 x'\right).
	\end{displaymath}
\end{prop}
\begin{proof}
	Since $\bp$ is a homomorphism,
	\begin{align*}
		\bp\left(\sum x b\1 \otimes x' b\2 \right) &= \bp \left((x \otimes x')\left(\sum b\1 \otimes b\2\right)\right) \\
		&= \bp(x\otimes x') \bp\left(\sum b\1 \otimes b\2\right) = (xx')b,
	\end{align*}
	where the last equality follows from (\ref{eq:zero_element}).
\end{proof}
To express the product of $A$ in terms of the products of $F(A)$ and $B$ we need some of the following maps from $A$ to $A$:
\begin{align} \label{eq:isotropy}
	 \nonumber l(b,b')u &:= \sum (b\1 b'\1) \backslash (b\2 \cdot b'\2 u),\\
	 \nonumber r(b,b')u &:= \sum (u b\1 \cdot b'\1) / (b\2 b'\2),\\
	  t(b)u &:= \sum b\1 \backslash (u b\2) ,\\
	 \nonumber s(b,b')u &:= \sum (b\1 \cdot u b'\1)/(b\2 b'\2), \\
	 \nonumber \bar{s}(b,b')u &:= \sum (b\1 b'\1) \backslash (b\2 u \cdot b'\2) \quad \text{and} \\
	 \nonumber \bar{t}(b)u &:= \sum (b\1 u)/b\2,
\end{align}
where $b, b' \in B$ and $u \in A$.
\begin{lem}
\label{lem:Prim-stable}
	$F(A)$ and $\Prim F(A)$ are invariant under all maps in (\ref{eq:isotropy}).
\end{lem}
\begin{proof}
	The projection of any element $x \in F(A)$ by $\pi$ is $\epsilon(x)1$. 
	We only show invariance of $F(A)$ under $l(b,b')$; the rest of the cases are proved similarly. 
	The projection of $l(b,b')x$ is $\epsilon(b)\epsilon(b')\epsilon(x)1 = \epsilon(l(b,b')x)1$. 
	Since $\Delta(l(b,b')x) = \sum l(b\1,b'\1)x\1 \otimes l(b\2,b'\2)x\2$, the image of $F(A)$ by $l(b,b')$ is a sub-coalgebra and 
	therefore, it is one of the summands in the definition of $F(A)$, i.e. $l(b,b')F(A) \subseteq F(A)$.
	Invariance of $\Prim F(A)$ is now obvious.
\end{proof}
The following proposition describes $A$ in terms of $F(A)$ and $B$.
\begin{prop}
\label{prop:product}
	Let $x, x' \in F(A)$ and $b, b' \in B$. We have that
	\begin{align*}
		xb \cdot x'b' &= \sum \left(r(b\1, b'\1)x \cdot s(b\2, b'\2)x'\right) \cdot b\3 b'\3 \quad \text{and}\\
		bx \cdot b'x' &= \sum b\1 b'\1 \cdot \left( \bar{s}(b\2, b'\2)x \cdot l(b\3,b'\3)x'\right).
	\end{align*}
\end{prop}
\begin{proof}
	The crucial point is that $\bp$ is a homomorphism:
	\begin{align*}
		&\sum \left(r(b\1, b'\1) x \cdot s(b\2, b'\2)x' \right)\cdot b\3 b'\3 \\
		& \quad \stackrel{\langle 1 \rangle}{=} \sum \bp\left( r(b\1,b'\1)x \cdot b\2 b'\2 \otimes s(b\3,b'\3)x' \cdot b\4 b'\4 \right)\\
		& \quad = \sum \bp \left( xb\1 \cdot b'\1 \otimes b\2 \cdot x'b'\2  \right)\\
		& \quad = \bp\left( \sum xb\1 \otimes b\2  \right) \bp\left( \sum b'\1 \otimes x'b'\2  \right)\\
		& \quad \stackrel{\langle 2 \rangle}{=} xb \cdot x'b'
	\end{align*}
	where $\langle 1 \rangle$ and $\langle 2 \rangle$ follow by Proposition \ref{prop:distributive}. 
	The second formula in the statement can be proved in a similar way.
\end{proof}
\begin{prop}
\label{prop:module-algebra}
	Let $\varphi \in \{l,r,s,\bar{s}\}$ denote a map from (\ref{eq:isotropy}). Then
	\begin{align*}
	 \varphi(b,b') (xx') &= \sum \varphi(b\1, b'\1)x \cdot \varphi(b\2, b'\2)x', \\
	 t(b) (xx') &= \sum t(b\1)x \cdot t(b\2)x' \quad \text{and} \\
	 \bar{t}(b) (xx') &= \sum \bar{t}(b\1)x \cdot \bar{t}(b\2)x'
	\end{align*}
	for any $x, x' \in F(A)$.
\end{prop}
\begin{proof}
	We have
	\begin{align*}
		& \sum b\1 b'\1 \cdot \left( l(b\2, b'\2)x \cdot l(b\3, b'\3)x' \right) \\
		& \quad \stackrel{\langle 1 \rangle}{=} \sum \bp\left( b\1 b'\1 \cdot l(b\2, b'\2)x \otimes b\3 b'\3 \cdot l(b\4, b'\4)x' \right)	\\
		& \quad = \sum \bp\left( b\1 \cdot b'\1 x \otimes b\2 \cdot b'\2 x' \right)\\
		& \quad = \sum \bp( b\1 \otimes b\2) \cdot \bp(b'\1 \otimes b'\2) \bp(x \otimes  x')\\
		& \quad = b (b'\cdot xx')
	\end{align*}
	where $\langle 1 \rangle$ follows from Proposition \ref{prop:distributive}. Therefore,
	\begin{displaymath}
		\sum l(b\1, b'\1) x \cdot l(b\2, b'\2)x' = l(b,b')(xx').
	\end{displaymath}
	This proves the case $\varphi = l$; the rest of the cases are proved using similar arguments.
\end{proof}
Some sort of associativity between $F(A)$ and $B$ also holds.
\begin{prop}
	Let $x,x' \in F(A)$ and $b, b' \in B$. We have that
	\begin{displaymath}
		x(x'b) = (xx')b, \quad (bx)x' = b(xx') \quad \text{and} \quad (xb)x' = x(bx')
	\end{displaymath}
\end{prop}
\begin{proof}
	The different formulas in the statement are a direct consequence of Proposition \ref{prop:product}. For example,
	\begin{displaymath}
		x(x'b) = \sum \left(r(1,b\1)x \cdot s(1,b\2)x'\right) b\3 = (xx')b.
	\end{displaymath}
\end{proof}
%
\subsubsection{Fundamental theorem on modules for objects in $\cB$}
Let $B$ be an object in $\cB$ and let $\Mult_B$ be the unital associative algebra generated by the set $\{ \llambda_b, \rrho_b \mid b \in B\}$ 
with relations
\begin{displaymath}
	\llambda_1 = 1 = \rrho_1, \quad \llambda_{\alpha b + \alpha' b'} = \alpha \llambda_b + \alpha' \llambda_{b'} 
	\quad \textrm{and} \quad 
	\rrho_{\alpha b + \alpha' b'} = \alpha \rrho_b + \alpha' \rrho_{b'}
\end{displaymath}
for any $\alpha, \alpha' \in \field$ and $b,b' \in B$. We consider the bialgebra structure on $\Mult_B$ determined by
\begin{displaymath}
	\Delta(\llambda_b) := \sum \llambda_{b\1} \otimes \llambda_{b\2}, \quad \Delta(\rrho_b):= \sum \rrho_{b\1} \otimes \rrho_{b\2}
\end{displaymath}
and
\begin{displaymath}
	\epsilon(\llambda_b):= \epsilon(b), \quad \epsilon(\rrho_{b}):= \epsilon(b).
\end{displaymath}
Since $B$ is irreducible, we can use induction on the coradical filtration of $B$ (see Section \ref{sec:Modules_Formal_Loops}) 
to prove the existence of uniquely determined elements $S(\llambda_b)$ and $S(\rrho_b)$ in $\Mult_B$ such that
\begin{displaymath}
	\sum S(\llambda_{b\1})\llambda_{b\2} = \epsilon(b) 1 = \sum S(\rrho_{b\1})\rrho_{b\2}.
\end{displaymath}
We extend $S$ to $\Mult_B$ by imposing that $S(\phi \phi') = S(\phi')S(\phi)$ for any $\phi, \phi'\in \Mult_B$.
\begin{prop}
	The bialgebra $\Mult_B$ is a Hopf algebra with antipode $S$.
\end{prop}
\begin{proof}
	We check that $\sum \llambda_{b\1}S(\llambda_{b\2}) = \epsilon(b)1 = \sum \rrho_{b\1}S(\rrho_{b\2})$ and 
	the statement follows by induction on the degree of the elements.
	If we consider $S'(\llambda_b)$ and $S'(\rrho_b)$ satisfying
	\begin{displaymath}
		\sum \llambda_{b\1}S'(\llambda_{b\2})= \epsilon(b)1 = \sum \rrho_{b\1}S'(\rrho_{b\2}),
	\end{displaymath}
	then we have that
	\begin{displaymath}
		S(\llambda_b)=\sum S(\llambda_{b\1})(\llambda_{b\2} S'(\llambda_{b\3})) 
		= 
		\sum (S(\llambda_{b\1})\llambda_{b\2}) S'(\llambda_{b\3})  = S'(\llambda_b).
	\end{displaymath}
\end{proof}
Given an abelian group $(A \stackrel{\pi}{\rightarrow} B, \bp, \bm, 0)$, the bialgebra $A$ is a (usual) left $\Mult_B$-module 
with the action determined by
\begin{displaymath}
 	\llambda_{b}u := 0(b)u = bu \quad \textrm{and} \quad \rrho_{b}u := u0(b) = ub
\end{displaymath}
for any $u \in A$. This action induces an arrow $\Mult_B \rightarrow B$ defined by $\phi \mapsto \phi 1$ in the category of coalgebras. 
We now determine the kernel of this arrow in $\cB$.
\begin{lem}
\label{lem:antipode_division}
	For any $b \in B$ and $u \in A$ we have that
	\begin{displaymath}
		S(\llambda_b)u = b\backslash u \quad \textrm{and} \quad S(\rrho_b)u = u/b
	\end{displaymath}
\end{lem}
\begin{proof}
	Since $\epsilon(b)u = \sum \llambda_{b\1}S(\llambda_{b\2})u = \sum b\1 S(\llambda_{b\2})u$, we have that $S(\llambda_{b})u = b\backslash u$. 
	Similarly, $S(\rrho_b) u = u/b$.
\end{proof}
To obtain an action of a Hopf algebra on $F(A)$ we have to consider a subalgebra of $\Mult_B$. 
For any $b, b' \in B$ we define the elements $\rr(b,b'), \sss(b,b') \in \Mult_B$ by
\begin{align*}
	\rr(b,b') &:= \sum S(\rrho_{b\1 b'\1}) \rrho_{b'\2} \rrho_{b\2}, \\
	\sss(b,b') &:= \sum S(\rrho_{b\1 b'\1}) \llambda_{b\2} \rrho_{b'\2}.
\end{align*}
Lemma \ref{lem:antipode_division} implies that for any $u \in A$
\begin{displaymath}
	\rr(b,b')u = r(b,b')(u) \quad \textrm{and} \quad \sss(b,b')u = s(b,b')(u),
\end{displaymath}
where the maps in the right-hand side of the equalities are those defined in (\ref{eq:isotropy}).

The subalgebra of $\Mult_B$ generated by $\{\rr(b,b'), \sss(b,b') \mid b,b' \in B\}$ will be denoted by $\Mult^+_B$. 
$\Mult^+_B$ is a Hopf subalgebra of $\Mult_B$; moreover, it is the largest Hopf subalgebra of $\Mult_B$ stabilizing $F(A)$. 
\begin{lem}
	We have that
	\begin{enumerate}
		\item $\Mult^+_B = \{ \sum S(\rrho_{\phi\1 1})\phi\2 \mid \phi \in \Mult_B\}$ and that
		\item $\rrho_b \otimes \phi \mapsto \rrho_b \phi$ defines an isomorphism of coalgebras $\Mult_B \cong \rrho_B \otimes \Mult^+_B$.
	\end{enumerate}
\end{lem}
\begin{proof}
	The proof of part (2) is similar to that of Proposition \ref{prop:factorization} so we omit it. 
	Given $\phi \in \Mult^+_B$, since $\phi 1 = \epsilon(\phi)1$ and $\Mult^+_B$ is a coalgebra
	we have that $\phi = \sum \epsilon(\phi\1)\phi\2 = \sum S(\rrho_{\phi\1 1})\phi\2$. 
	Thus we only have to prove that for any $\phi \in \Mult_B$, the element $\sum S(\rrho_{\phi\1 1})\phi\2$ belongs to $\Mult^+_B$. 
	We proceed by induction on the degree of $\phi$ on $\{\llambda_b, \rrho_b \mid b \in \ker \epsilon \subseteq B \}$.
	Note that the case $\phi =1$ is trivial. 
	If we consider $\phi = \llambda_b$, we obtain $\sum S(\rrho_{\phi\1 1})\phi\2 = \sss(b,1)$.
	Taking $\phi = \rrho_b$ we get $\sum S(\rrho_{\phi\1 1})\phi\2 = \epsilon(b)1$. 
	In general, for $\phi = \llambda_b \phi'$ we have
	\begin{displaymath}
		\sum S(\rrho_{\phi\1 1})\phi\2 = \sum \sss(b,\phi'\1 1) S(\rrho_{\phi'\2 1}) \phi'\3
	\end{displaymath}
	and for $\phi = \rrho_b \phi'$ we have
	\begin{displaymath}
		\sum S(\rrho_{\phi\1 1})\phi\2 = \sum \rr(\phi'\1 1, b)S(\rrho_{\phi'\2 1})\phi'\3,
	\end{displaymath}
	which proves the induction step.
\end{proof}
Lemma \ref{lem:Prim-stable} and Proposition \ref{prop:module-algebra} imply that $\Prim F(A)$ is a (usual) left $\Mult^+_B$-module.

Summarizing, we obtained the following results about the abelian group $(A \stackrel{\pi}{\rightarrow} B, \bp, \bm, 0)$ in $\cB \downarrow B$.
\begin{enumerate}
	\item $F(A)$ is isomorphic to the symmetric algebra on $\Prim F(A)$.
	\item $\Prim F(A)$ is a left $\Mult^+_B$-module.
	\item This action extends to an action on $F(A)$ by Proposition \ref{prop:module-algebra}.
	\item $A \cong F(A) \otimes B$ as bialgebras with the product given by Proposition \ref{prop:product}.
\end{enumerate}
\begin{thrm} \label{thm:fundamental}
	Let $B$ be an object in $\cB$. The category of $B$-modules and the category of (usual) unital left $\Mult^+_B$-modules are equivalent.
\end{thrm}
\begin{proof}
We only have to prove that for any unital left $\Mult^+_B$-module $V$ we can construct an abelian group 
$( A\stackrel{\pi}{\rightarrow} B, \bp, \bm, 0)$ in $\cB \downarrow B$ such that $\Prim F(A) = V$ and 
that the action of $\Mult^+_B$ on $\Prim F(A)$ agrees with the action on $V$.

Given a unital left $\Mult^+_B$-module $V$, we extend the action of $\Mult^+_B$ to $\field[V]$ by 
$\phi 1 := \epsilon(\phi)$ and $\phi (xx') := \sum (\phi\1 x)(\phi\2 x')$ for any $x,x' \in \field[V]$. 
Now define $$A:= \field[V] \otimes B$$ with the coalgebra structure of the tensor product and multiplication given by
\begin{displaymath}
	(x\otimes b)(x' \otimes b'):= \sum (\rr(b\1,b'\1)x) (\sss(b\2,b'\2)x') \otimes b\3 b'\3.
\end{displaymath}
With these operations, $A$ is an object in $\cB$. 
The projection $A \stackrel{\pi}{\rightarrow} B$ is defined by $x \otimes b \mapsto \epsilon(x)b$. 
Any subcoalgebra of $A$ containing $\field[V] \otimes 1$ contains a primitive element of $B$ (it must be connected) 
so $\field[V] \otimes 1$ is the largest subcoalgebra on which $\pi$ and $\epsilon 1$ agree, i.e. $F(A) = \field[V] \otimes 1$. 
Hence $\Prim F(A) = V \otimes 1$.

In $A$ we have that
\begin{align*}
	(x\otimes 1)(1\otimes b) &= x \otimes b \\
	((1\otimes b)(x \otimes 1))(1\otimes b') &= \sum \sss(b\1,b'\1)x \otimes b\2 b'\2,
\end{align*}
hence the maps $r(b,b'), s(b,b')$ defined in (\ref{eq:isotropy}) agree on $\field[V]\otimes 1$ with the action of $\rr(b,b'), \sss(b,b') \in \Mult^+_B$, 
after identifying $\field[V] \otimes 1$ with $\field[V]$. 
This shows that the action of $\Mult^+_B$ on $\Prim F(A)$ is the same as the action on $V$ after identifying $V \otimes 1$ with $V$.

The following map determines a morphism in $\cB \downarrow B$
\begin{align*}
0\colon B &\rightarrow A\\
b &\mapsto 1 \otimes b.
\end{align*}
Since the action of $\Mult^+_B$ on $\field[V]$ preserves the homogeneous components, the map
\begin{align*}
\bm \colon A &\rightarrow A\\
x\otimes b & \mapsto S(x)\otimes b,
\end{align*}
where $S$ denotes the antipode of $\field[V]$, is a homomorphism of bialgebras which induces a morphism in $\cB \downarrow B$. 
By Proposition \ref{prop:F_AxA}, $A \otimes_B A = \spann\langle \sum x \otimes b\1 \otimes x' \otimes b\2 \mid x, x' \in \field[V], b \in B\rangle $, 
so we can define a map
\begin{align*}
\bp \colon A \otimes_B A &\rightarrow A\\
\sum x \otimes b\1 \otimes x' \otimes b\2 &\mapsto xx' \otimes b
\end{align*}
which induces a morphism in $\cB \downarrow B$ (beware that this implies that $\bp$ has to be a homomorphism of bialgebras). 
Finally, it is not difficult to check that $(A \stackrel{\pi}{\rightarrow} B, \bp, \bm, 0)$ is an abelian group in $\cB \downarrow B$.
\end{proof}

%
%
\section{Modules for Moufang loops revisited} \label{sec:Moufang_modules}
The representation theory of Moufang loops as exposed in \cite{Dharwadker1995} is a particular case of modules for loops in a variety. 
In this section we adopt a new approach based on the idea of Moufang elements in arbitrary loops. 
A \emph{Moufang element} in a loop $E$ is an element $a \in E$ satisfying
\begin{displaymath}
	a(x(ay)) = ((ax)a)y \quad \text{and} \quad ((xa)y)a = x(a(ya))
\end{displaymath}
for all $x,y \in E$ \cite{Phillips2009}. 
Note that $E$ might not be a Moufang loop, but Moufang elements are not effected by it. 
The set of all Moufang elements of a loop $E$ will be denoted by $M(E)$ and is always a Moufang loop. 
This means that a Moufang loop might strongly exhibit its nature even inside loops that are not Moufang.

\subsection{Relative modules for Moufang loops}
Given a Moufang loop $Q$, abelian groups in the comma category $\cL \downarrow Q$ of loops over $Q$ do not globally reflect the Moufang symmetry of $Q$. 
However, an abelian group in the comma category $\cM \downarrow Q$ of Moufang loops over $Q$ is subject to many restrictions 
that have nothing to do with $Q$ because we impose that all elements, not only those in $Q$, behave as Moufang elements. 
A compromise solution is to consider
\begin{definition}
	A \emph{relative representations} or \emph{relative module} of a Moufang loop $Q$ is an abelian group 
	$(E \stackrel{\pi}{\rightarrow} Q, \bp, \bm, 0)$  in $\cL \downarrow Q$ with $0(Q) \subseteq M(E)$.
\end{definition}
Obviously, usual representations of groups and Moufang loops are relative representations too. 
The proof of the following result is straightforward.
\begin{prop} \label{prop:relative_equations}
	Let $Q$ be a Moufang loop. We have that an abelian group $(E \stackrel{\pi}{\rightarrow} Q, \bp, \bm, 0)$ in $\cL \downarrow Q$ 
	is a relative representation of $Q$ if and only if the maps $r(a,b)$ and $s(a,b)$ defined in (\ref{eq:rs}) satisfy
	\begin{align*}
		r(a, c(bc)) &= r((ac)b,c)r(ac,b)r(a,c),\\
		r((ac)b,c) s(ac,b) &= s(a,c(bc)) s(c,bc) r(b,c),\\
		s(c,a(cb)) r(a,cb) &= r((ca)c,b)r(a,c)s(c,a) \\
		s((ca)c,b) &= s(c,a(cb))s(a,cb)s(c,b)
	\end{align*}
	for any $a,b,c \in Q$.
\end{prop}
However, there is a much simpler description of these new representations.  
Consider $\Doro(Q)$, the group generated by $\{\lambda_a,\rho_a, \tau_a  \mid a \in Q \}$ subject to relations
\begin{displaymath}
	\begin{array}{llll}
		\lambda_e = 1, & \rho_e = 1, & \tau_e = 1, & \tau_a \lambda_a \rho_a = 1,\\
		\lambda_{aba} = \lambda_a\lambda_b\lambda_a,
		 & \rho_{aba} = \rho_a\rho_b\rho_a, & \tau_{aba} = \tau_a\tau_b\tau_a, &\\
		 \lambda_{a^{-1}b} = \tau_a\lambda_b\rho_a, & \rho_{a^{-1}b} = \lambda_a\rho_b\tau_a, & \tau_{a^{-1}b} = \rho_a \tau_b \lambda_a,\\
		 \lambda_{ba^{-1}} = \rho_a \lambda_b \tau_a, & \rho_{ba^{-1}} = \tau_a \rho_b \lambda_a, & \tau_{ba^{-1}} = \lambda_a\tau_b \rho_a. &	
	\end{array}
\end{displaymath}
This group has a venerable history in the theory of Moufang loops \cites{Glauberman1968,Doro1978,Mikheev1993,Grishkov2006,Hall2010,Benkart2013} 
and it is related with the idea of \emph{triality}. 
The following proposition is essentially the motivation to define $\Doro(Q)$ \cites{Doro1978,Glauberman1968}, so we omit the proof.
\begin{prop}
\label{prop:action}
	Let $Q$ be a Moufang loop and $E$ a loop such that $Q \subseteq M(E)$. There exists an action of $\Doro(Q)$ on $E$ determined by
\begin{displaymath}
\lambda_a x := ax, \quad \rho_a x := xa \quad \text{ and } \quad \tau_a x := (a\backslash x) /a.
\end{displaymath}
\end{prop}
However, the group which best describes relative representations of a Moufang loop $Q$ is the following subgroup of $\Doro(Q)$:
\begin{displaymath}
	\Doro(Q)_e := \textrm{ subgroup of } \Doro(Q) \textrm{ generated by } 
	\left\{ \rho^{-1}_{ab}\rho_b\rho_a \, , \, \rho^{-1}_{ab}\lambda_a\rho_b \mid a, b \in Q \right\}.
\end{displaymath}

\begin{prop}
Relative representations of a Moufang loop $Q$ are equivalent to $\Doro(Q)_e$-modules.
\end{prop}
\begin{proof}
Since for any relative representation $(E \stackrel{\pi}{\rightarrow} Q, \bp, \bm, 0)$ we have $0(Q) \subseteq M(E)$,
$\Doro(Q)$ acts on $E$ as in Proposition \ref{prop:action}. 
$E_e$ also becomes a $\Doro(Q)_e$-module under this action. 
Conversely, starting with a $\Doro(Q)_e$-module $E_e$, we consider the loop $E = E_e \times Q$ defined by (\ref{eq:Smith}), 
where in this case
\begin{displaymath}
	r(a,b) := \rho^{-1}_{ab} \rho_b \rho_a \qquad s(a,b) := \rho^{-1}_{ab}\lambda_a\rho_b
\end{displaymath}
belong to $\Doro(Q)_e$. 
The action of $\U(Q;\cL)$ on $E$ factors through $\Doro(Q)$, so the induced action of $\U(Q;\cL)_e$ on $E_e$ factors through $\Doro(Q)_e$. 
The elements $r(a,b), s(a,b)$ in (\ref{eq:rs}) act on $E_e$ as the elements $r(a,b), s(a,b) \in \Doro(Q)_e$ we just defined. 
Moreover, we can use the relations in $\Doro(Q)$ to check that $Q \subseteq M(E)$ (recall that $Q$ is identified with $0 \times 0(Q)$). 
Thus we only have to check the relations in Proposition \ref{prop:relative_equations}. 
We write in detail the proof of the first one; the rest follow by similar computations.
We have that
\begin{align*}
	r((ac)b,c)s(ac,b) &= \rho^{-1}_{((ac)b)c}\rho_c\rho_{(ac)b}\rho^{-1}_{(ac)b} \lambda_{ac}\rho_b \\
	&=  \rho^{-1}_{((ac)b)c}\rho_c \lambda_{ac}\rho_b = \rho^{-1}_{a(c(bc))}\rho_c \lambda_{ac}\rho_b \\[2pt]
	s(a,c(bc))s(c,bc)r(b,c) &= \rho^{-1}_{a(c(bc))}\lambda_a\rho_{c(bc)}\rho^{-1}_{c(bc)}\lambda_c\rho_{bc}\rho^{-1}_{bc} \rho_c \rho_b\\
	&= \rho^{-1}_{a(c(bc))}\lambda_a\lambda_c\rho_c \rho_b,
\end{align*}
so we need to check that $\rho_c\lambda_{ac} = \lambda_a \lambda_c \rho_c$. 
This is a direct consequence of the relations $\lambda_{yx^{-1}} = \rho_x \lambda_y \tau_x$, $\lambda_{x^{-1}} = \lambda^{-1}_x$, 
$\rho_{x^{-1}} = \rho^{-1}_x$ and $\tau_x = \rho^{-1}_x\lambda^{-1}_x$ in $\Doro(Q)$.
\end{proof}
\subsubsection*{Examples}
Consider a group $G$ and the \emph{group of autotopisms} of $G$ \cite{Benkart2013}
\begin{displaymath}
	\Atp(G) := \{ (\phi_1,\phi_2,\phi_3) \mid \phi_1(ab) = \phi_2(a)\phi_3(b) \;\forall_{a,b \in G}\}.
\end{displaymath}
Examples of elements in $\Atp(G)$ are $(L_a, L_a, \Id), (R_a, \Id, R_a), (\Id, R^{-1}_a,L_a)$. 
There exists a homormorphism $\Doro(G) \rightarrow \Atp(G)$ (see \cite{Benkart2013}) determined by
\begin{displaymath}
	\lambda_a \mapsto (L_a,P^{-1}_a,L^{-1}_a) \qquad \rho_a \mapsto (R_a, R^{-1}_a,P^{-1}_a).
\end{displaymath} 
The image of $\Doro(G)_e$ under this homomorphism lies in the subgroup generated by
$(L_aR^{-1}_a,L_a,R^{-1}_a)$ and $(\Id, R^{-1}_b,L^{-1}_b)$, 
which is isomorphic to $G \times G$ by
\begin{displaymath}
	(a,b) \mapsto (L_bR^{-1}_b, L_bR^{-1}_a,R^{-1}_bL_a).
\end{displaymath}
Therefore, via the homomorphism $\Doro(G) \rightarrow \Atp(G)$, any module $V$ for $G \times G$ can be viewed as a module for $\Doro(G)_e$. 
The image of $r(a,b)$ and $s(a,b)$ are the autotopisms
\begin{displaymath}
	(\Id,R^{-1}_{b^{-1}a^{-1}ba}, L_{b^{-1}a^{-1}ba}) \quad \text{and} \quad (L_aR^{-1}_a,L_a,R^{-1}_a)(\Id,R^{-1}_{b^{-1}a^{-2}b}, L_{b^{-1}a^{-2}b})
\end{displaymath}
respectively. 
Hence, $r(a,b), s(a,b)$ act on $V$ as $ (b^{-1}a^{-1}ba,e)$ and $(b^{-1}a^{-2}b,a)$ respectively, and we obtain a relative representation of $G$.
\begin{prop}
	Let $G$ be a group and $V, W$ be linear representations of $G$. The set $E = V \otimes W \times G$ with product
	\begin{displaymath}
		\left(\sum v_i \otimes w_i,a\right) \hskip -3pt  \left(\sum v'_j \otimes w'_j,b\right) \hskip -2pt  
		= 
		\hskip -2pt \left(\sum b^{-1}a^{-1}bav_i \otimes w_i \hskip -1pt  + \hskip -1pt  \sum b^{-1}a^{-2}b v'_j \otimes a w'_j,ab \right)
	\end{displaymath}
	is a loop such that $G \cong 0\otimes 0 \times G \subseteq M(E)$.
\end{prop}
Notice that if $G$ is simple non-abelian and $V$ is faithful then $G \cap \Na(E) = \{ e \}$, 
where $\Na(E) := \{a \in E \mid (ay)z = a(yz),\, (xa)z = x(az),\, (xy)a = x(ya)\,\, \forall x,y,z \in E \}$ 
denotes the \emph{associative nucleus} of $E$.

\subsection{Relative modules for formal Moufang loops}
Now we will extend the notion of relative module to a formal setting. 
Recall that given a non-associative algebra $A$, the \emph{generalized alternative nucleus} of $A$ is defined as
\begin{displaymath}
	\Nalt(A):= \{ a \in A \mid (a,y,z) = -(y,a,z) = (y,z,a) \quad \forall y,z \in A \}
\end{displaymath}
where $(x,y,z) := (xy)z - x(yz)$ denotes the \emph{associator} of $x,y$ and $z$ \cites{Morandi2001,Perez-Izquierdo2004}. 
$\Nalt(A)$ is closed under the commutator product $[a,b]:= ab - ba$ and it is a \emph{Malcev algebra} with this product.

A \emph{formal Moufang loop} is a formal loop $F \colon \field[\m \times \m] \rightarrow \m$ satisfying the identities
\begin{displaymath}
	\xx (\yy(\xx \zz)) = ((\xx \yy)\xx)\zz \qquad ((\zz\xx)\yy)\xx = \zz (\xx(\yy \xx)).
\end{displaymath}
In other words, the bialgebra $\field[\m]$ with product $zz' := F'(z\otimes z')$ satisfies
\begin{displaymath}
	\sum z\1 (u (z \2 u)) = \sum ((z\1u)z\2)v \qquad \sum ((vz\1)u)z\2 = \sum v(z\1(u z\2)),
\end{displaymath}
so $\m$ is a Malcev algebra with the commutator product and $\field[\m]$ is isomorphic to $U(\m)$, 
the universal enveloping algebra of $\m$ \cites{Perez-Izquierdo2004,Perez-Izquierdo2007}, 
an algebra with a universal property with respect to homomorphisms $\m \rightarrow \Nalt(A)$ of Malcev algebras.
\begin{definition}
	A \emph{relative module} for the formal Moufang loop $F\colon \field[\m \times \m] \rightarrow \m$ is an abelian group $( G \stackrel{\pi}{\rightarrow} F, \bp, \bm, 0)$ in the comma category of formal loops over $F$ satisfying
	\begin{displaymath}
	 	0(\xx) (\yy(0(\xx) \zz)) = ((0(\xx) \yy)0(\xx))\zz, \qquad ((\zz 0(\xx))\yy)0(\xx) = \zz (0(\xx)(\yy 0(\xx))),
	\end{displaymath}
	where $0(\xx)$ denotes the composition of the formal maps $0$ and $\xx$.
\end{definition}
\begin{definition}
	A \emph{relative module} for a Moufang Hopf algebra $B$ is an abelian group in $(A\stackrel{\pi}{\rightarrow} B, \bp, \bm, 0)$ 
	in $\cB \downarrow B$ satisfying
	 \begin{align}
		\label{eq:relative1}	\sum b\1 (u (b \2 v)) &= \sum (( b\1 u) b\2)v \\
		\label{eq:relative2}	\sum ((v b\1)u) b\2 &= \sum v( b\1(u  b\2))
	\end{align}
	for any $b \in B$ and $u, v \in A$.
\end{definition}
Thus, thanks to the equivalence of categories of formal loops and irreducible bialgebras, 
we have that the category of relative modules for a formal Moufang loop $F\colon \field[\m \times \m] \rightarrow \m$ 
is equivalent to the category of relative modules for the Moufang Hopf algebra $U(\m)$.
\begin{prop} \label{prop:Nalt}
	An abelian group $(A\stackrel{\pi}{\rightarrow} U(\m), \bp, \bm, 0)$ in $\cB \downarrow U(\m)$ 
	is a relative module for $U(\m)$ if and only if $0(\m) \subseteq \Nalt(A)$.
\end{prop}
\begin{proof}
The proof is similar to the proof of \cite[Theorem 14]{Perez-Izquierdo2007} so we omit it.
\end{proof}
\subsubsection{Relative modules for $U(\m)$ are $U(\Lie(\m)_+)$-modules} \label{subsect:relative_modules}
The goal in \cite{Perez-Izquierdo2004} was to understand whether any Malcev algebra can be constructed as a 
Malcev subalgebra of $\Nalt(A)$  for some non-associative algebra $A$. 
A natural construction in this context is the Lie algebra $\Lie(\m)$ generated by abstract symbols 
$\lambda_a, \rho_a$ $a \in \m$ subject to relations
\begin{displaymath}
	\begin{array}{ll}
		\lambda_{\alpha a + \alpha' a'} = \alpha \lambda_a + \alpha' \lambda_{a'}, 
		& \rho_{\alpha a + \alpha' {a'}} = \alpha \rho_a + \alpha' \rho_{a'}, 
		\cr
		[\lambda_a,\lambda_{a'}] = \lambda_{[a,{a'}]} - 2 [\lambda_a,\rho_{a'}], 
		& [\rho_a, \rho_{a'}] = -\rho_{[a,{a'}]} - 2[\lambda_a,\rho_{a'}], 
		\cr
		[\lambda_a, \rho_{a'}] = [\rho_a,\lambda_{a'}] &
	\end{array}
\end{displaymath}
for all $\alpha,\alpha' \in \field$ and $a,{a'} \in \m$.
This algebra $\Lie(\m)$ models the action of the left and right multiplication operators $L_a, R_a$ by elements $a \in \Nalt(A)$ \cite{Morandi2001}. 
In the sight of Proposition \ref{prop:Nalt}, this Lie algebra must play a relevant role here too. 
The most useful construction in our setting of relative modules for $U(\m)$ is
\begin{displaymath}
	\Lie(\m)_+ :=\textrm{ Lie subalgebra of } \Lie(\m) \textrm{ generated by } \{ \ad_a:= \lambda_a - \rho_a \mid a \in \m\}.
\end{displaymath}
For any relative module $(A\stackrel{\pi}{\rightarrow} U(\m), \bp, \bm, 0)$ the homomorphism $\Mult_{U(\m)} \rightarrow \Endo(A)$ 
determined by $\llambda_z \mapsto L_z$ and $\rrho_z \mapsto R_z$ factors through $U(\Lie(\m))$ 
and its restriction to $\Mult^+_{U(\m)}$ factors through $U(\Lie(\m)_+)$. 
Hence, relative modules of $U(\m)$ are $\Lie(\m)_+$-modules.

Consider $\{(((a_{i_1}a_{i_2})\cdots )a_{i_n} \mid i_1 \leq i_2\leq \cdots \leq i_n, n \geq 0\}$, a Poincar\'e-Birkhoff-Witt basis of $U(\m)$ 
\cite{Perez-Izquierdo2004}, and define the following elements in $U(\Lie(\m))$
\begin{align}
	\lambda_{((a_{i_1}a_{i_2})\cdots) a_{i_n}} 
	&:= \lambda_{((a_{i_1}a_{i_2})\cdots )a_{i_{n-1}}} \lambda_{a_{i_n}} + [\lambda_{((a_{i_1}a_{i_2})\cdots )a_{i_{n-1}}},\rho_{a_{i_n}}]
	\cr
	\rho_{((a_{i_1}a_{i_2})\cdots) a_{i_n}} 
	&:= \rho_{a_{i_n}} \rho_{((a_{i_1}a_{i_2})\cdots )a_{i_{n-1}}} + [\lambda_{a_{i_n}},  \rho_{((a_{i_1}a_{i_2})\cdots )a_{i_{n-1}}}]
\end{align}
with $\lambda_1 := 1$ and $\rho_1 := 1$, and extend them by linearity to elements $\lambda_z, \rho_z \in U(\Lie(\m))$. 
With this definition $\lambda_z$ and $\rho_z$ act as the left and right multiplication operators by $z$ on $U(\m)$ \cite{Morandi2001}. 
Thus we have a commuting diagram
\begin{center}
	\begin{tikzpicture}
		\node(ll) at (-1,0.7) {$\Mult_{U(\m)}$};
		\node(rr) at (1,0.7) {$\Endo(A)$};
		\node(dd) at (0,-0.7) {$U(\Lie(\m))$};
		
		\draw[->] ($(ll.east)$) -- node[above]{} ($(rr.west)$);
		\draw[->] ($(ll.south)$) -- node[left]{} ($(dd.north)$);
		\draw[->] ($(dd.north)$) -- node[right]{}($(rr.south)$);
	\end{tikzpicture}
\end{center}
where the homomorphism $\Mult_{U(\m)} \rightarrow U(\Lie(\m))$ is determined by $\llambda_z \mapsto \lambda_z$ and $\rrho_z \mapsto \rho_z$. 
Prior to checking that this homomorphism restricts to a homomorphism $\Mult^+_{U(\m)} \rightarrow U(\Lie(\m)_+)$ 
we give some properties of the elements $\lambda_z, \rho_z$ we introduced.

There is an automorphism $\sigma$ of $U(\Lie(\m))$ \cite{Perez-Izquierdo2004} determined on the generators by
\begin{displaymath}
	\lambda_a \mapsto - \rho_a, \qquad \rho_a \mapsto - \lambda_a.
\end{displaymath}
The composition of this automorphism with the antipode $S$ of $U(\Lie(\m))$ gives an anti-automorphism $\sigma S$ 
interchanging $\lambda_x$ and $\rho_x$.
\begin{prop} \label{prop:lambdaRecursive}
	For any $z \in U(\m)$ and $a \in \m$ we have that
	\begin{alignat*}{2}
		\lambda_{za} = \lambda_z \lambda_a + [\lambda_z, \rho_a], \qquad
		\rho_{za} = \rho_a \rho_z + [\lambda_a, \rho_z],
		\\
		\lambda_{az} = \lambda_a \lambda_z + [\rho_a,\lambda_z], \qquad
		\rho_{az} = \rho_z \rho_a + [\rho_z,\lambda_a].
	\end{alignat*}
\end{prop}
\begin{proof}
The second identity is obtained from the first one by applying the anti-automorphism $\sigma S$. 
The third and fourth identities share the same relation. 
Thus, we only need to show the first and the third identities. 
We prove them by induction on the filtration degree $\vert z \vert$ of $z$ (see \cite{Perez-Izquierdo2004}) that for any $b \in \m$
\begin{align}
\label{eq:lambda} \lambda_{zb} &= \lambda_z \lambda_b + [\lambda_z, \rho_b] \\
\label{eq:lambdaComm}\lambda_{[z,b]} &= [\lambda_z, \lambda_b + 2 \rho_b],	
\end{align}
the case $\vert z \vert = 1$ being trivial. 
The identity $\lambda_{az} =\lambda_a \lambda_z + [\rho_a,\lambda_z]$ follows easily from these two identities.

Note that elements of the form $[z,c] \in U(\m)$ with $c \in \m$ have filtration degree $\leq \vert z \vert$. 
We assume that these equalities hold for elements $z$ with $\vert z \vert < n$, 
and show them for $z = ((a_{i_1}a_{i_2})\cdots) a_{i_n}$. 
Let $z':= ((a_{i_1}a_{i_2})\cdots)a_{i_{n-1}}$ and $a:= a_{i_n}$. We have
\begin{align*}
	[\lambda_{{z'}a},\lambda_b + 2 \rho_b] &= [\lambda_{z'}\lambda_a + [\lambda_{z'},\rho_a], \lambda_b + 2\rho_b]\cr
	&= \lambda_{z'}[\lambda_a,\lambda_b + 2 \rho_b] + [\lambda_{z'}, \lambda_b + 2 \rho_b]\lambda_a + [[\lambda_{z'},\rho_a],\lambda_b + 2 \rho_b]\cr
	&= \lambda_{z'} \lambda_{[a,b]} + \lambda_{[{z'},b]}\lambda_a + [[\lambda_{z'},\rho_a],\lambda_b + 2 \rho_b]\cr
	&\stackrel{\langle 1 \rangle}{=} \lambda_{z'} \lambda_{[a,b]} + \lambda_{[{z'},b]}\lambda_a + [\lambda_{[{z'},b]},\rho_a] + [\lambda_{z'}, [\rho_a, \lambda_b + 2 \rho_b]] \cr
	&\stackrel{\langle 2 \rangle}{=} \lambda_{z'} \lambda_{[a,b]} + \lambda_{[{z'},b]a} + [\lambda_{z'},[\rho_a, \lambda_b + 2 \rho_b]] \cr
	&\stackrel{\langle 3 \rangle}{=} \lambda_{{z'}[a,b]} - [\lambda_{z'},\rho_{[a,b]}] + \lambda_{[{z'},b]a} + [\lambda_{z'},[\rho_a, \lambda_b + 2\rho_b]] \cr
	&= \lambda_{{z'}[a,b]+[{z'},b]a} + [\lambda_{z'}, -\rho_{[a,b]} + [\rho_a, \lambda_b + 2\rho_b]]\cr
	&\stackrel{\langle 4 \rangle}{=} \lambda_{{z'}[a,b]+[{z'},b]a} - [\lambda_{z'} , 3 (\rho_{[a,b]} + [\rho_a,\lambda_b])]
\end{align*}
where $\langle 1 \rangle - \langle 3 \rangle$ follow from the hypothesis of induction and $\langle 4 \rangle$ follows from the relations in $\Lie(\m)$. 
The element $3 (\rho_{[a,b]} + [\rho_a,\lambda_b])$ can be written in terms of elements $\lambda_c + 2 \rho_c$:
\begin{align*}
	[\lambda_a + 2\rho_a, \lambda_b + 2 \rho_b] &= \lambda_{[a,b]} - 4 \rho_{[a,b]} - 6 [\lambda_a, \rho_b] \cr
	&= \lambda_{[a,b]} + 2 \rho_{[a,b]} - 6(\rho_{[a,b]} + [\lambda_a,\rho_b])
\end{align*}
so
\begin{equation} \label{eq:trick}
	3(\rho_{[a,b]} + [\lambda_a, \rho_b]) = \frac{1}{2}(\lambda_{[a,b]} + 2\rho_{[a,b]}) - \frac{1}{2}[\lambda_a + 2\rho_a, \lambda_b + 2 \rho_b].
\end{equation}
Then, we can write $[\lambda_{{z'}a}, \lambda_b + 2\rho_b]$ as
\begin{align*}
	[\lambda_{{z'}a}, \lambda_b + 2\rho_b] 
	&= 
	\lambda_{{z'}[a,b] + [{z'},b]a} - [\lambda_{z'}, \frac{1}{2}(\lambda_{[a,b]} 
	+ 2\rho_{[a,b]}) - \frac{1}{2}[\lambda_a + 2\rho_a, \lambda_b + 2 \rho_b]].
\end{align*}
Using the Jacobi identity and the hypothesis of induction we can conclude that $[\lambda_{{z'}a},\lambda_b + 2 \rho_b] = \lambda_z$ for some $z \in U(\m)$.
The action of this element on $1$ gives $z = ({z'}a)(3b) -b({z'}a)-2({z'}a)b = [{z'}a,b]$. 
One can also check directly 
that $z = {z'}[a,b] +[{z'},b]a - \frac{1}{2} [{z'},[a,b]] + \frac{1}{2}[[{z'},a],b] - \frac{1}{2}[[{z'},b],a] = [{z'}a,b]$ in $U(\m)$. 
This proves (\ref{eq:lambdaComm}).

To prove (\ref{eq:lambda}) we can assume that $b$ is a basic element $a_{i_{n+1}}$. 
In the case that $a_{i_n} \leq a_{i_{n+1}}$, (\ref{eq:lambda}) follows from the very definition of $\lambda_{zb}$. 
Thus, we may assume that $a_{i_{n+1}} < a_{i_n}$, i.e., $b < a$. We have that
\begin{align*}
	\lambda_{zb} - \lambda_z\lambda_b - [\lambda_z,\rho_b] &= \lambda_{({z'}a)b} - \lambda_{{z'}a}\lambda_b -[\lambda_{{z'}a},\rho_b] 
	\cr
	&\stackrel{\langle 1 \rangle}{=} 
	\lambda_{({z'}a)b} - \lambda_{z'} \lambda_a \lambda_b - [\lambda_{z'},\rho_a] \lambda_b - [\lambda_{z'}\lambda_a + [\lambda_{z'},\rho_a],\rho_b] 
	\cr
	&\stackrel{\langle 2 \rangle}{=} 
	\lambda_{({z'}a)b-({z'}b)a} + \lambda_{{z'}b}\lambda_a + [\lambda_{{z'}b},\rho_a] - \lambda_{z'}\lambda_a\lambda_b - [\lambda_{z'},\rho_a]\lambda_b 
	\cr
	& \quad  
	- [\lambda_{z'}\lambda_a + [\lambda_{z'},\rho_a],\rho_b]
	\cr
	&\stackrel{\langle 3 \rangle}{=} 
	\lambda_{({z'}a)b-({z'}b)a} + \lambda_{z'}\lambda_b\lambda_a + [\lambda_{z'},\rho_b]\lambda_a + [\lambda_{z'}\lambda_b +[\lambda_{z'},\rho_b],\rho_a]
	\cr
	& \quad - \lambda_{z'}\lambda_a\lambda_b - [\lambda_{z'},\rho_a]\lambda_b - [\lambda_{z'}\lambda_a + [\lambda_{z'},\rho_a],\rho_b] 
	\cr
	&\stackrel{\langle 4 \rangle}{=} 
	\lambda_{({z'}a)b - ({z'}b)a} + \lambda_{z'} [\lambda_b,\lambda_a] + \lambda_{z'} [\lambda_b,\rho_a] 
	\cr
	& \quad 
	- \lambda_{z'} [\lambda_a,\rho_b] + [\lambda_{z'},[\rho_b,\rho_a]]
	\cr
	&\stackrel{\langle 5 \rangle}{=} 
	\lambda_{({z'}a)b - ({z'}b)a} - \lambda_{z'} \lambda_{[a,b]} - [\lambda_{z'},[\rho_a,\rho_b]] 
	\cr
	&\stackrel{\langle 6 \rangle}{=} 
	\lambda_{({z'}a)b - ({z'}b)a - {z'}[a,b]} + [\lambda_{z'}, \rho_{[a,b]} - [\rho_a,\rho_b]]
	\cr
	&\stackrel{\langle 7 \rangle}{=} 
	\lambda_{({z'}a)b - ({z'}b)a - {z'}[a,b]} + 2 [\lambda_{z'}, \rho_{[a,b]} + [\lambda_a,\rho_b]],
\end{align*}
where $\langle 1 \rangle$, $\langle 3 \rangle$, $\langle 6 \rangle$ and $\langle 7 \rangle$ follow from induction, 
$\langle 4 \rangle$ follows from the Jacobi identity, $\langle 5 \rangle$ from the defining identities of $\Lie(\m)$ and 
$\langle 2 \rangle$ is a consequence of the definition of the symbols $\lambda_z$ and $\rho_z$ since $b <a$. 
By (\ref{eq:trick}) and the Jacobi identity we can write the latter equality as $\lambda_z$ for some $z \in U(\m)$. 
The action of $\lambda_{zb} - \lambda_z \lambda_b - [\lambda_z,\rho_b]$ on $1$ gives that $z = 0$. This proves (\ref{eq:lambda}).
\end{proof}
The universal enveloping algebra $U(\m)$ also has an antipode $S$ \cites{Perez-Izquierdo2004,Perez-Izquierdo2007} 
which is an involutive anti-automorphism satisfying $z\backslash z' = S(z)z'$ and $z'/z = z'S(z)$ for any $z,z' \in U(\m)$. 
The antipodes of $U(\Lie(\m))$ and $U(\m)$ are nicely related:
\begin{prop}
	For any $z \in U(\m)$ we have that
	\begin{displaymath}
		S(\lambda_z) = \lambda_{S(z)}, \qquad S(\rho_z) = \rho_{S(z)}.
	\end{displaymath}
\end{prop}
\begin{proof}
We use induction in the filtration degree $\vert z \vert$ of $z$, the case $\vert z \vert  = 0$ being trivial. 
The general case works, using Proposition \ref{prop:lambdaRecursive}, as follows
\begin{align*}
	 S(\lambda_{za}) &= S(\lambda_z \lambda_a + [\lambda_z,\rho_a]) = S(\lambda_a)S(\lambda_z) + [S(\rho_a),S(\lambda_z)]\cr
	 &= \lambda_{S(a)}\lambda_{S(z)}  + [\rho_{S(a)},\lambda_{S(z)}] = \lambda_{S(a)S(z)} = \lambda_{S(za)},
\end{align*}
where $a \in \m$.
\end{proof}
\begin{lem}
\label{lem:lambdaDelta}
	For any $z \in U(\m)$ we have that
	\begin{displaymath}
		\Delta(\lambda_z) = \sum \lambda_{z\1} \otimes \lambda_{z\2}, \qquad \Delta(\rho_z) = \sum \rho_{z\1} \otimes \rho_{z\2}.
	\end{displaymath}
\end{lem}
\begin{proof}
The proof is easily obtained using Proposition \ref{prop:lambdaRecursive} and induction on the filtration degree of $\vert z\vert$.
\end{proof}
Therefore, by Lemma \ref{lem:lambdaDelta} the homomorphism $\Mult_{U(\m)} \rightarrow U(\Lie(\m))$ determined by 
$\llambda_z \mapsto \lambda_z$ and $\rrho_z \mapsto \rho_z$ is a homomorphism of bialgebras. 
Moreover, it is a homomorphism of Hopf algebras by the recursive definition of $S(\llambda_z)$ and $S(\rrho_z)$ and the fact that 
$\sum S(\lambda_{z\1})\lambda_{z\2} = \epsilon(z)1 = \sum S(\rho_{z\1})\rho_{z\2}$.

Given $z,z' \in U(\m)$ we define the elements
\begin{equation}
\label{eq:l}
	r(z,z') := \sum S(\rho_{z\1 z'\1})\rho_{z'\2}\rho_{z\2}
\end{equation}
and
\begin{equation}
\label{eq:r}
s(z,z') := \sum S(\rho_{z\1 z'\1}) \lambda_{z\2} \rho_{z'\2}.
\end{equation}
Notice that
\begin{displaymath}
	 r(1,z) = r(z,1) = s(1,z) = \epsilon(z)1.
\end{displaymath}
\begin{prop}
	For any $z,z' \in U(\m)$ $r(z,z')$ and $s(z,z')$ belong to $U(\Lie(\m)_+)$.
\end{prop}
\begin{proof}
By Lemma \ref{lem:lambdaDelta}, 
$\spann \, \langle \, r(z,z') \mid z,z' \in U(\m) \, \rangle$ and $\spann \,  \, s(z,z') \mid z,z' \in U(\m) \, \rangle$ 
are subcoalgebras of $U(\Lie(\m))$ and they are subsets of $\{\phi \in U(\Lie(\m))\mid \phi 1 = \epsilon(\phi)1\}$, 
where $\phi 1$ denotes the action of $\phi$ on $1 \in U(\m)$. 
Thus, since the sum of subcoalgebras is a subcoalgebra, it will be enough to prove that $U(\Lie(\m)_+)$ is the largest subcoalgebra 
contained in that subspace. 
Any subcoalgebra of $U(\Lie(\m))$ larger than $U(\Lie(\m)_+)$ contains a primitive element which is not in $\Lie(\m)_+$, 
so it contains an element of the form $\lambda_a$ for some nonzero $a \in \m$. 
However, $\lambda_a 1 = a \neq 0 = \epsilon(\lambda_a)1$. This proves the result.
\end{proof}
\begin{thrm} \label{thm:main}
The homomorphism $\Mult_{U(\m)} \rightarrow U(\Lie(\m))$ determined by $\llambda_z \mapsto \lambda_z$ and $\rrho_z \mapsto \rho_z$ 
restricts to a homomorphism $\Mult^+_{U(\m)} \rightarrow U(\Lie(\m)_+)$ of Hopf algebras sending 
$\rr(z,z')$ to $r(z,z')$ and $\sss(z,z')$ to $s(z,z')$ for all $z,z' \in U(\m)$.
\end{thrm}
Thus, any relative module of a formal Moufang loop $F \colon \field[\m \times \m] \rightarrow \m$ is determined by its structure as $\Lie(\m)_+$-module. 
In the next section we will prove that any $\Lie(\m)_+$-module integrates to a relative module of the corresponding formal Moufang loop.
%
%
\section{Modules for Malcev algebras revisited} \label{Modules_for_Malcev_algebras_revisited}
The representation theory of Malcev algebras \cites{Kuzmin1968,Carlsson1976,Elduque1990,Elduque1995} is modeled 
by the notion of split-null extension in the variety of Malcev algebras, 
so it characterizes split-null extensions of local Moufang loops in the variety of Moufang loops. 
It fails to describe relative representations of Moufang loops in the sense of Section \ref{sec:Moufang_modules}. 
A new infinitesimal counterpart of these representations is required. 
In this section we define such representations of Malcev algebras and integrate them to representations of formal Moufang loops. 
All over this section $\m$ will denote a Malcev algebra.

\subsection{Modules for Malcev algebras}
The representation theory of Malcev algebras has been beautifully developed by Kuzmin, Carlsson, Elduque and Shestakov among others. 
The translation of Kuzmin's fundamental paper \emph{Structure and representations of finite dimensional Malcev algebras} 
by Tvalavadze, edited by Bremner and Madariaga \cite{Kuzmin2014}, contains a brief survey of recent developments. 
Here we review some results needed to put our novel approach to this topic into perspective. 
We keep the traditional way of making operators to act on the right when citing results.

A \emph{representation} of a Malcev algebra $\m$ is a linear map $\rho \colon \m \rightarrow \Endo_{\field}(V)$ such that
\begin{displaymath}
	\rho_{[[x,y],z]} = \rho_x\rho_y\rho_z - \rho_z \rho_x\rho_y + \rho_y\rho_{[z,x]}
\end{displaymath}
or, equivalently, if $V \oplus \m$ with product
\begin{displaymath}
	(v+x)(w+y) = v \rho_y - w \rho_x + [x,y]
\end{displaymath}
is a Malcev algebra. In this case the $\field$-vector space $V$ is called a \emph{module} for $\m$. 
Unfortunately, irreducible modules for Malcev algebras are very scarce.
\begin{thrm}[Carlsson, \cite{Carlsson1976}] \label{thm:Carlsson}
Over fields $\field$ of characteristic zero
\begin{itemize}
	\item[(a)] Any irreducible module for $\sl2(2,\field)$ regarded as Malcev algebra is either a module for $\sl2(2,\field)$ regarded as Lie algebra 
		or a $2$-dimensional module with basis $\{v,w\}$ such that if $\{e,f,h\}$ is a basis for $\sl2(2,\field)$ 
		with $[eh] = e, [f,h] = -f$ and $[ef] = \frac{1}{2}h$, then the action of $\sl2(2,\field)$ is given by
		\begin{displaymath}
			v\cdot h = v, \  w \cdot h = -w, \  v\cdot e = w, \  v \cdot f = 0, \  w \cdot e = 0, \  w \cdot f = -v.
		\end{displaymath}
		This module is said to be of type $M_2$.
	\item[(b)] Any irreducible module for the $7$-dimensional simple central Malcev algebra $\Oc_0$ is either trivial or isomorphic to the adjoint module.
\end{itemize}
\end{thrm}
This result was extended by Elduque to fields of characteristic $\neq 2,3$ \cite{Elduque1990} and later to arbitrary dimension \cite{Elduque1995}. 
The two-dimensional non-Lie representation of $\sl2(2,\field)$ in Theorem \ref{thm:Carlsson} seem to be very exceptional in this approach 
to the representation theory of Malcev algebras. However, it is interesting to note that this exceptional representations is isomorphic to
\begin{displaymath}
	v\cdot x := -2 vx
\end{displaymath}
for any $v \in \field \times \field$ and $x \in \sl2(2,\field)$, where $xv$ denotes the usual matrix product. 
This description is in full accordance with Theorem \ref{thm:modules} which shows how this module is no longer exceptional 
as a relative representation of the Malcev algebra $\sl2(2,\field)$, but a member of a unified series of such representations 
appearing for any Lie algebra, not only for $\sl2(2,\field)$.

\subsection{Relative modules for Malcev algebras}
We now introduce two new Lie algebras $\Lie(\m)_+$ and $\Lie(\m)$ and show that they are isomorphic to those previously denoted by these symbols, 
so the reader will find no ambiguity in this notation. 
The approach to $\Lie(\m)_+$ through generators and relations is justified to properly define the notion of relative module for Malcev algebras.
\begin{definition}
	Let $\m$ be a Malcev algebra. A \emph{relative representation} of $\m$ is a linear map $l \colon \m \rightarrow \Endo(V)$ satisfying
	\begin{displaymath}
		[[l_a,l_b],l_c] = -[l_{[a,b]},l_c] + l_{[[a,b],c] + [[a,c],b] + [a,[b,c]]}
	\end{displaymath}
		for any $a,b,c \in \m$.
\end{definition}
\begin{prop}
	Let $\Lie(\m)_+$ be the Lie algebra generated by symbols $\{\ad_a \mid a \in \m\}$ subject to relations
	\begin{itemize}
		\item[a)] $\ad_{\alpha a + \beta b} = \alpha \ad_a + \beta \ad_b$ for all $\alpha,\beta \in \field$ and $a,b \in \m$,
		\item[b)] $[[\ad_a,\ad_b],\ad_c] + [\ad_{[a,b]},\ad_c] - \ad_{[[a,b],c] + [[a,c],b] + [a,[b,c]]}$ for all $a,b,c \in \m$.
	\end{itemize}
	Then, the category of relative representations of $\m$ is equivalent to the category of representations of the Lie algebra $\Lie(\m)_+$.
\end{prop}
Note that the symbol $\ad_a$ can be specialized to the usual adjoint map $\ad_a \colon b \mapsto [a,b]$ of $\m$, 
obtaining a relative representation of $\m$ (the \emph{adjoint representation}), 
so there is no confusion when using $\ad_a$ as an abstract generator of $\Lie(\m)_+$.

For each pair of elements $a, b \in \m$, we define
\begin{displaymath}
	D_{a,b} := \frac{1}{2}(\ad_{[a,b]} + [\ad_a, \ad_b]) \in \Lie(\m)_+.
\end{displaymath}
Consider $T_\m$ a copy of $\m$ whose elements are denoted by $T_a$ with $a \in \m$ and 
note that $T_{\alpha a + \beta b} = \alpha T_a + \beta T_b$ for any $\alpha, \beta \in \field$ and $a,b \in \m$.
\begin{prop} \label{prop:L}
	The vector space $\Lie(\m) := \Lie(\m)_+ \oplus T_\m$ is a Lie algebra with the product determined by
	\begin{align*}
		[\ad_a,T_x] &:= T_{[a,x]},\cr
		[[\ad_a,\ad_b],T_x] &:= T_{[a,[b,x]] - [b,[a,x]]} \cr
		[T_a,T_b] &:= \frac{1}{3} (\ad_{[a,b]}  + 2D_{a,b}) = \frac{1}{3}(2\ad_{[a,b]} + [\ad_a,\ad_b]).
	\end{align*}	
\end{prop}
\begin{proof}
We should check that the cyclic sum of the product of three elements in $ ad_{\m} \cup D_{\m,\m} \cup T_\m $ is zero. 
Since $\Lie(\m)_+$ is a Lie algebra, we may assume that at least one element belongs to $T_\m$. 
We first deal with the case where the three elements belong to $\ad_{\m} \cup \, T_\m$.  
Let $J := [[T_a,T_b],T_c] + [[T_b,T_c],T_a] +[[T_c,T_a],T_b]$. We have
\begin{align*}
		J &= \frac{1}{3} \Big( [2\ad_{[a,b]}+[\ad_a,\ad_b],T_c] + [2\ad_{[b,c]}+[\ad_b,\ad_c],T_a] 
		\cr
		& \quad + [2\ad_{[c,a]}+[\ad_c,\ad_a],T_b] \Big) 
		\cr
		&= \frac{1}{3} \Big( 2 T_{[[a,b],c]} + T_{[a,[b,c]] - [b,[a,c]]} +2 T_{[[b,c],a]} + T_{[b,[c,a]] - [c,[b,a]]} 
		\cr
		& \quad +2 T_{[[c,a],b]} + T_{[c,[a,b]] -[a,[c,b]]}  \Big) = 0.
\end{align*}
Now let $J := [[T_a,T_b],\ad_x] + [[T_b,\ad_x],T_a]+ [[\ad_x,T_a],T_b]$. We have
\begin{align*}
		J &= \frac{1}{3}[ 2\ad_{[a,b]}+[\ad_a,\ad_b], \ad_x ] + [T_{[b,x]},T_a] + [T_{[x,a]},T_b] \cr
		&= \frac{1}{3} \Big( 2[\ad_{[a,b]},\ad_x] - [\ad_{[a,b]},\ad_x] + \ad_{[[a,b],x]+[[a,x],b]+[a,[b,x]]}
		\cr
		&  \quad +2 \ad_{[[b,x],a]} +[\ad_{[b,x]},\ad_a] +2 \ad_{[[x,a],b]} + [\ad_{[x,a]},\ad_b] \Big)
		\cr
		&= \quad \frac{1}{3} \Big(  [\ad_{[a,b]},\ad_x] + [\ad_{[b,x]},\ad_a] + [\ad_{[x,a]},\ad_b] 
		\cr
		& \quad + \ad_{[[b,x],a]+[[x,a],b]+[[a,b],x]} \Big) = 0,
\end{align*}
where the last equality follows from the fact that by construction $\Lie(\m)_+$ satisfies the Jacobi identity and 
from the relations between the generators of $\Lie(\m)_+$. 
Finally, consider $J:=[[T_x,\ad_a],\ad_b] + [[\ad_a,\ad_b],T_x] + [[\ad_b,T_x],\ad_a]$. Again,
\begin{align*}
		J &= [T_{[x,a]},\ad_b] - T_{[[x,a],b]-[[x,b],a]} + [T_{[b,x]},\ad_a] 
		\cr
		&= T_{[[x,a],b]} + T_{[a,[b,x]] - [b,[a,x]]} + T_{[[b,x],a]} = 0.
\end{align*}

To deal with the case where elements in $D_{\m,\m}$ appear in the cyclic sum we define the map $D_{a,b} \colon \m \rightarrow \m$ given by
\begin{equation}
\label{eq:D}
	 D_{a,b}(x) := \frac{1}{2} \Big( [[a,b],x] + [a,[b,x]]-[b,[a,x]] \Big).
\end{equation}
This map is a derivation of $\m$ \cite{Sagle1961} and it is related to $D_{a,b} \in \Lie(\m)_+$ by
\begin{align}
	& [D_{a,b},T_x] = T_{D_{a,b}(x)} , \qquad [D_{a,b},\ad_x] = \ad_{D_{a,b}(x)},	\label{eq:DT-Dad} \\
	& [D_{a,b},D_{x,y}]=D_{D_{a,b}(x),y} + D_{x,D_{a,b}(y)}. \label{eq:DD}
\end{align}
The Jacobi identity for the remaining cases follows from these relations.
\end{proof}
\begin{prop} \label{prop:isomorphism}
	The Lie algebra $\Lie(\m)$ is isomorphic to the Lie algebra $\Lie(\m)$ defined in Section \ref{subsect:relative_modules}.
\end{prop}
\begin{proof}
To avoid ambiguities in this proof, we denote by $\Lie(\m)'$ the Lie algebra defined in \ref{subsect:relative_modules} 
and use $\Lie(\m)$ for the Lie algebra in Proposition \ref{prop:L}.

By Proposition \ref{prop:L}, the elements $\lambda_a  := \tfrac{1}{2}\ad_a + \tfrac{1}{2} T_a, \rho_a := \tfrac{1}{2}\ad_a - \tfrac{1}{2} T_a \in \Lie(\m)$ 
satisfy the defining relations of $\Lie(\m)'$ so we have a homomorphism $\phi \colon \Lie(\m)' \rightarrow \Lie(\m)$ 
sending $\lambda_a$ to $\lambda_a$ and $\rho_a$ to $\rho_a$.

Conversely, the elements $\ad_a \in \Lie(\m)'$ satisfy the defining relations of $\Lie(\m)_+$, 
so we have a homomorphism $\Lie(\m)^+ \rightarrow \Lie(\m)'$ sending $\ad_a$ to $\ad_a$. 
By Proposition \ref{prop:L} and the formulas for the product on $\Lie(\m)'$, this homomorphism can be extended 
to a homomorphism $\phi' \colon \Lie(\m) \rightarrow \Lie(\m)'$ sending $\ad_a$ to $\ad_a$ and $T_a$ to $T_a$. 
Clearly $\phi'$ is the inverse of $\phi$.
\end{proof}
%
%
\subsection{Relative modules for \fd central simple Malcev algebras}
%
Proposition \ref{prop:isomorphism} shows that any relative module of a formal Moufang loop $F \colon \field[\m \times \m] \rightarrow \m$ 
is determined by its structure as a relative module for the Malcev algebra $\m$. 
Before proving that relative representations of Malcev algebras can be formally integrated to relative representations of the corresponding formal loops, 
we classify the relative modules of finite-dimensional (f.d.) central simple Malcev algebras to show how the new theory 
extends the usual theory of representations of Malcev algebras.

Any \fd central simple Malcev algebra is either a simple Lie algebra or an algebra of traceless octonions with the commutator product. 
So to develop our theory we first consider Lie algebras.

\begin{prop}
	Let $\m$ be a Lie algebra. Then
	\begin{align*}
	 	I_M &:= \spann \langle \, \ad_{[a,b]} - D_{a,b} \mid a, b \in \m \, \rangle \\
	 	I_L &:= \spann \langle \, \ad_{[a,b]} + 2 D_{a,b} \mid a, b \in \m \, \rangle
	\end{align*}
	are ideals of $\Lie(\m)_+$ and $[I_M,I_L] = 0$.
\end{prop}
\begin{proof}
We use the following relation in $\Lie(\m)_+$
\begin{displaymath}
	[[\ad_a,\ad_b],\ad_c] = 2\ad_{[[a,b],c]} - [\ad_{[a,b]},\ad_c],
\end{displaymath}
which can also be written as
\begin{displaymath}
	[D_{a,b},\ad_c] = \ad_{[[a,b],c]}.
\end{displaymath}
Let $\xi_{a,b}:=\ad_{[a,b]} - D_{a,b}$ and $\xi'_{c,d} := \ad_{[c,d]} + 2 D_{c,d} = 2 \ad_{[c,d]} + [\ad_c,\ad_d]$. We have
\begin{align*}
	[\xi_{a,b}, \ad_c] &= [\ad_{[a,b]} - D_{a,b},\ad_c] = 2D_{[a,b],c}-\ad_{[[a,b],c]} - \ad_{[[a,b],c]} = -2 \xi_{[a,b],c} 
	\\
	[\xi'_{a,b}, \ad_c] &= [\ad_{[a,b]} + 2 D_{a,b},\ad_c] = 2D_{[a,b],c} - \ad_{[[a,b],c]} + 2\ad_{[[a,b],c]} = \xi'_{[a,b],c} .
\end{align*}
Since $\Lie(\m)_+$ is generated by $\ad_\m$, this proves that $I_M$ and $I_L$ are ideals. Finally
\begin{align*}
	 [\xi_{a,b},\xi'_{c,d}] &= [\xi_{a,b},2\ad_{[c,d]} + [\ad_c,\ad_d] ]
	 -4\xi_{[a,b],[c,d]} - 2[\xi_{[a,b],c},\ad_d] - 2[\ad_c, \xi_{[a,b],d}]\cr
	 &= - 4\xi_{[a,b],[c,d]} + 4 \xi_{[[a,b],c],d} +4 \xi_{c,[[a,b],d]}.
\end{align*}
The vanishing of this element is equivalent to the identity
\begin{displaymath}
	[\ad_{[a,b]},\ad_{[c,d]}] = [\ad_{[[a,b],c]},\ad_d] + [\ad_c,\ad_{[[a,b],d]}], 
\end{displaymath}
which can be proved as follows
\begin{align*}
	[\ad_{[a,b]},\ad_{[c,d]}] &= 2 \ad_{[[a,b],[c,d]]} - [\ad_{[a,b]},[\ad_c,\ad_d]] 
	\cr
	&= 2 \ad_{[[a,b],[c,d]]} - [[\ad_{[a,b]},\ad_c],\ad_d] - [\ad_c,[\ad_{[a,b]},\ad_d]] 
	\cr
	&= 2 \ad_{[[a,b],[c,d]]} -2\ad_{[[[a,b],c],d]} -2 \ad_{[c,[[a,b],d]]} 
	\cr
	& \quad + [\ad_{[[a,b],c]},\ad_d] + [\ad_c,\ad_{[[a,b],d]}]
	\cr
	&= [\ad_{[[a,b],c]},\ad_d] + [\ad_c,\ad_{[[a,b],d]}].
\end{align*}
\end{proof}
\begin{prop} \label{prop:2m}
	Let $\m$ be a \fd semisimple Lie algebra. Then assigning
	\begin{displaymath}
		\ad_a \mapsto (-2a,a)
	\end{displaymath}
	induces an isomorphism of Lie algebras $\Lie(\m)_+ \cong \m \times \m$.
\end{prop}

\begin{proof}
Let $l_a := (-2a,a)$. We have that
\begin{displaymath}
	[[l_a,l_b],l_c] = -[l_{[a,b]},l_c] + l_{[[a,b],c]+[[a,c],b] +[a,[b,c]]}
\end{displaymath}
so there exists a homomorphism of Lie algebras $\Lie(\m)_+ \rightarrow \m \times \m$ induced by $\ad_a \mapsto l_a$. 
We check that it is an isomorphism. 
The image of $D_{a,b}$ is $([a,b],[a,b])$ so the image of $\ad_{[a,b]}-D_{a,b}$ is $(-3[a,b],0)$ and 
the image of $\ad_{[a,b]} + 2 D_{a,b}$ is $(0,3[a,b])$. 
Since $\m$ is semisimple, we obtain the surjectivity. 
If $\ad_a + \sum_i D_{a_i,b_i}$ belongs to the kernel of the homomorphism, then $(-2a+\sum_i[a_i,b_i],a+\sum_i[a_i,b_i]) = (0,0)$, 
so $a = 0 = \sum_i[a_i,b_i]$. 
Therefore, the kernel consists of all the elements $\sum_i D_{a_i,b_i}$ such that $\sum_i[a_i,b_i] = 0$. 
Since $[D_{a,b},\ad_c] = \ad_{[[a,b],c]}$, the kernel is the center of $\Lie(\m)_+$ and it coincides with the radical of $\Lie(\m)_+$. 
The Levi decomposition of $\Lie(\m)_+$ shows that in this case the intersection of the center and the derived algebra $[\Lie(\m)_+,\Lie(\m)_+]$ is zero. 
However, $\sum_i[a_i,b_i] = 0$ implies that $\sum_i[\ad_{a_i},\ad_{b_i}] = 2 \sum_i D_{a_i,b_i}$ 
and so the center is contained in $[\Lie(\m)_+,\Lie(\m)_+]$. 
This proves that the kernel is zero.
\end{proof}
Now we can describe the relative modules for semisimple Lie algebras. 
The term \emph{Lie module} refers to a usual module for a Lie algebra as opposed to the more general kind of modules that we are considering in this paper.

\begin{thrm} \label{thm:modules}
	Let $\m$ be a \fd semisimple Lie algebra and $V$ an irreducible relative module of $\m$. 
	Then there exist irreducible Lie modules $V_M$ and $V_L$ of $\m$ such that $V$ is isomorphic to the vector space $V_M \otimes V_L$ with action
	\begin{displaymath}
		a * v_M \otimes v_L := -2(av_M) \otimes v_L + v_M \otimes (av_L)
	\end{displaymath}
	for any $a \in \m$, $v_M \in V_M$ and $v_L \in V_L$.
\end{thrm}

\begin{proof}
Relative modules are the same as $\Lie(\m)_+$-modules. Since we know that $\ad_a \mapsto (-2a,a)$ induces 
an isomorphism between $\Lie(\m)_+$ and $\m \times \m$ then irreducible relative modules of $\m$ correspond to 
irreducible $\Lie(\m)_+$-modules where the element $a \in \m$ acts as $(-2a,a)$, i.e. 
to tensor products $V_M \otimes V_L$ of two irreducible Lie modules of $\m$ where $a$ acts as in the statement.
\end{proof}
Notice that in case that $V_M \cong \field$ is a trivial Lie module, $V_M \otimes V_L \cong V_L$ is a Lie module. 
However, when $V_L \cong \field$ is a trivial Lie module, $\m \cong \sl2(2,\field)$ and $V_M$ is its two-dimensional irreducible representation, 
we get the irreducible non-Lie Malcev module of $\sl2(2,\field)$. 
Therefore, in general, a relative module of a \fd semisimple Lie algebra is some sort of combination of a Lie module $V_L$ 
and a `purely Malcev' module $V_M$.

Relative representations of \fd non-Lie central simple Malcev algebras are easily derived. 
These algebras are known to be isomorphic to the traceless octonions with the commutator product. 
Thus after extending scalars we get the split octonion algebra. 
Since $\bar{\field}\otimes_{\field} \Lie(\m)_+ \cong \Lie(\bar{\field} \otimes_{\field} \m)_+$ for any field $\bar{\field}$ extending $\field$, 
the following result proves that the representation theory of these Malcev algebras corresponds to the representation theory 
of central simple Lie algebras of type $B_3$.
\begin{thrm}
	Let $\m$ be a \fd non-Lie central simple Malcev algebra. 
	Then $\Lie(\m)_+$ is a central simple Lie algebra of type $B_3$ and relative modules for $\m$ 
	are the same that $\Lie(\m)_+$-modules, the action being given by $a* v := \ad_a v$ for any $a \in \m$.
\end{thrm}
\begin{proof}
We only need to prove that $\Lie(\m)_+$ is a central simple Lie algebra of type $B_3$. 
We can assume that $\field$ is algebraically closed. 
The map $\ad \colon \m \rightarrow \gl(\m)$ $a \mapsto \ad_a$ defines a relative representation of $\m$ and 
induces an epimorphism between $\Lie(\m)_+$ and the multiplication Lie algebra $\Lie$ of $\m$, i.e. 
the Lie subalgebra of $\gl(\m)$ generated by the maps $\ad_a \colon x \mapsto [a,x]$, 
which is known to be a central simple Lie algebra of type $B_3$ \cite{Schafer1966}. 
The kernel of this epimorphism consists of the symbols $\ad_a + \sum_i D_{a_i,b_i} \in \Lie(\m)_+$ acting trivially on $\m$ 
($\ad_a$ acts as the adjoint map and $D_{a,b}$ acts as the derivation defined in (\ref{eq:D})). 
However, no non-zero derivation of $\m$ is of the form $\ad_a$, so $a = 0$. 
So we can conclude that $\sum_i D_{a_i,b_i} = 0$ as in the proof of Proposition \ref{prop:2m} 
(although a dimension counting argument, valid in positive characteristic $\neq 2,3$, is also possible), so the statement is proved.
\end{proof}
\subsection{Formal integration of relative modules for Malcev algebras}
Any relative module $V$ of a Malcev algebra $\m$ is a module for $U(\Lie(m)_+)$ and, by Theorem \ref{thm:main}, also for $\Mult^+_{U(\m)}$. 
Thus, by Theorem \ref{thm:fundamental}, it defines an abelian group $\left(\field[V] \otimes U(\m) \stackrel{\pi}{\rightarrow}U(\m), \bp, \bm, 0 \right)$ 
in $\cB \downarrow U(\m)$. By Proposition \ref{prop:Nalt} we only have to check that $\m \subseteq \Nalt(\field[V] \otimes U(\m))$.
\begin{thrm}
	Let $\m$ be a Malcev algebra and $V$ a relative module of $M$. Then $\field[V] \otimes U(\m)$ with the product given by the formula
	\begin{displaymath}
		(x \otimes b)(x' \otimes b') := \sum \left( r(b\1,b'\1)x \right) \left(s(b\2,b'\2)x' \right)\otimes b\3 b'\3
	\end{displaymath}
	for any $x, x' \in \field[V]$ and $b,b' \in U(\m)$, where $r(b,b')$ and $s(b,b')$ are defined by (\ref{eq:r}) and (\ref{eq:l}) respectively, 
	is a unital irreducible bialgebra whose generalized alternative nucleus contains $\m \cong 1\otimes \m$. 	
\end{thrm}
\begin{proof}
Note that the identification $\field[V] \otimes 1 \cong \field[V]$ and $1 \otimes U(\m) \cong U(\m)$ implies
\begin{displaymath}
	xb' = (x \otimes 1)(1 \otimes b') = \sum\left( r(1,b'\1)x \right) \left(s(1,b'\2) 1\right) \otimes b'\3 = x \otimes b',
\end{displaymath}
so we can safely remove the symbol $\otimes$ from our notation. 
Note also that $\field[V]$ and $U(\m)$ are subalgebras of $\field[V] \otimes U(\m)$, so no extra symbols for their products are needed. 
The fact that $\field[V] \otimes U(\m)$ is a unital bialgebra is an easy consequence of: 
1) $\Delta (\phi x) = \sum \phi\1 x\1 \otimes \phi\2 x\2$ for any $\phi \in U(\Lie(\m)_+)$ and $x \in \field[V]$, 
and 2) $\Delta(r(b,b')) = \sum r(b\1,b'\1) \otimes r(b\2,b'\2)$, $\Delta(s(b,b')) = \sum s(b\1, b'\1) \otimes s(b\2,b'\2)$. 
The space of primitive elements of $\field[V] \otimes U(\m)$ is $V \otimes \m$ 
because the space of primitive elements of $U(\m)$ (resp. $\field[V]$) is $\m$ (resp. $V$).

Therefore, we only have to check that $\m \subseteq \Nalt(\field[V]\otimes U(\m))$. 
To avoid a large number of parentheses we use again the symbol $\cdot$ to denote the product. 
For example, $xb\cdot x'b'$ represents the element $(xb)(x'b')$. 
We write the formula for the product on $\field[V] \otimes U(\m)$ as
\begin{displaymath}
	xb \cdot x'b' = \sum \left(r(b\1,b'\2)x \cdot s(b\2, b'\2)x' \right)\cdot b\3b'\3.
\end{displaymath}
We have to prove that
\begin{displaymath}
	(a,x'b',x''b'') = -(x'b',a,x''b'') = (x'b',x''b'',a)
\end{displaymath}
for any $x',x'' \in \field[V]$, $b',b'' \in U(\m)$ and $a \in \m$, where as usual $(x,y,z)$ denotes the associator. 
We only prove the first equality, leaving the rest to the reader.
We have
\begin{align*}
(&a ,x'b' ,x''b'') + (x'b',a,x''b'') \\
&= \sum \left( r(a\1 b'\1, b''\1)s(a\2,b'\2)x' \cdot s(a\3 b'\3, b''\2)x'' \right) \cdot (a\4 b'\4 )b''\3\\
&  \quad - \sum \left(s(a\1, b'\1 b''\1)r(b'\2, b''\2)x' \cdot s(a\2,b'\3 b''\3)s(b'\4, b''\4)x''\right) \cdot a\3(b'\5 b''\5)\\
&  \quad + \sum \left(r(b'\1 a\1, b''\1) r(b'\2, a\2)x' \cdot s(b'\3 a\3, b''\2)x''\right) \cdot (b'\4 a\4) b''\3 \\
&  \quad - \sum \left( r(b'\1 , a\1 b''\1)x' \cdot s(b'\2, a\2 b''\2)s(a\3, b''\3)x'' \right) \cdot b'\3(a\3 b''\4).
\end{align*}
Since $a$ is primitive we get $(a ,x'b', x''b'') + (x'b',a,x''b'') = \Sigma_1 - \Sigma_2 + \Sigma_3 - \Sigma_4$, where
\begin{align*}
\Sigma_1&:=  \sum \left( r(ab'\1, b''\1)s(1,b'\2)x' \cdot s(b'\3, b''\2)x'' \right) \cdot b'\4 b''\3 \\
& \quad + \sum \left( r(b'\1 , b''\1)s(a, b'\2) x' \cdot s(b'\3, b''\2)x'') \right) \cdot b'\4 b''\3 \\
& \quad + \sum \left( r(b'\1, b''\1)s(1,b'\2)x' \cdot s(ab'\3, b''\2)x''  \right) \cdot b'\4 b''\3 \\
& \quad + \sum \left( r(b'\1, b''\1)s(1, b'\2)x' \cdot s(b'\3, b''\2)x''   \right) \cdot (ab'\4)b''\3, \\[4pt]
\Sigma_2&:= \sum \left( s(a,b'\1 b''\1)r(b'\2,b''\2)x' \cdot s(1,b'\3 b''\3)s(b'\4, b''\4)x''\right) \cdot b'\5 b''\5 \\
& \quad + \sum \left( s(1, b'\1 b''\1 )r(b'\2, b''\2)x' \cdot s(a, b'\3 b''\3)s(b'\4, b''\4)x'' \right) \cdot b'\5 b''\5 \\
& \quad + \sum \left( s(1, b'\1 b''\1)r(b'\2, b''\2)x' \cdot s(1, b'\3 b''\3)s(b'\4, b''\4)x'' \right) \cdot a(b'\5 b''\5), \\[4pt]
\Sigma_3&:=  \sum \left( r(b'\1 a, b''\1)r(b'\2,1)x' \cdot s(b'\3, b''\2)x''  \right)\cdot b'\4 b''\3 \\
& \quad +  \sum \left( r(b'\1, b''\1)r(b'\2,a)x' \cdot s(b'\3, b''\2)x''  \right) \cdot b'\4 b''\3 \\
& \quad + \sum \left( r(b'\1, b''\1)r(b'\2,1)x' \cdot s(b'\3 a,b''\2)x''  \right) \cdot b'\4 b''\3 \\
& \quad + \sum \left( r(b'\1,b''\1)r(b'\2,1)x' \cdot s(b'\3,b''\2)x'' \right) \cdot (b'\3 a)b''\4 \\[4pt]
\Sigma_4&:=  \sum \left( r(b'\1, a b''\1)x' \cdot s(b'\2, b''\2)s(1,b''\3)x'' \right) \cdot b'\3 b''\4 \\
& \quad + \sum \left( r(b'\1,b''\1)x' \cdot s(b'\2, a b''\2)s(1,b''\3)x''  \right) \cdot b'\3 b''\4\\
& \quad + \sum \left( r(b'\1, b''\1)x' \cdot s(b'\2, b''\2)s(a, b''\3)x''  \right) \cdot b'\3 b''\4  \\
& \quad + \sum \left( r(b'\1,b''\1)x' \cdot s(b'\2,b''\2)s(1,b''\3)x'' \right) \cdot b'\3(a b''\4).
\end{align*}
Since $a$ belongs to the generalized alternative nucleus of $U(\m)$, the terms whose right factor include $a$ vanish. 
Cocommutativity of $\Delta$ allows us to write the right factor of the remaining terms as $b'\1 b''\1$, 
so we can ignore it when checking the vanishing of the sum.
We use ocommutativity of $\Delta$ again and reorder the reduced terms so either $r(b'\1,b''\1)x'$ or $s(b'\1,b''\1)x''$ is a common factor. 
Therefore, to prove that $(a ,x'b', x''b'') + (x'b',a,x''b'') = 0$ we only have to prove that
\begin{align} \label{eq:naltOne}
	S_1 & := r(ab',b'') + r(b'a,b'') -r(b',ab'')+ \sum  r(b'\1, b'') s(a, b'\2) \\
 		& \quad - s(a,b'\1 b''\1)r(b'\2,b''\2) +r(b'\1,b'')r(b'\2,a), \nonumber \\
	S_2 & := s(ab',b'') + s(b'a,b'') - s(b',ab'') \\
	    & \quad - \sum  s(a,b'\1 b''\1)s(b'\2, b''\2) + s(b',b''\1)s(a,b''\2)  = 0.  \nonumber
\end{align}
Since these equalities involve elements in $\U(\Lie(\m))$ we can use the definition of $r$ and $s$ in terms of $\lambda$ and $\rho$. 
First we prove (\ref{eq:naltOne}). After expanding it we get
\begin{align*}
S_1 &= \sum S(\rho_{b'\1 b''\1})\rho_{b''\2}\rho_{b'\2} \left( S(\rho_{ab'\3})\rho_{b'\4} + S(\rho_{b'\3}) \lambda_a \rho_{b'\4} \right)\\
&\quad  + S(\rho_{ab'\1 \cdot b''\1}) \rho_{b''\2}\rho_{b'\2} + S(\rho_{b'\1 b''\1}) \rho_{b''\2}\rho_{ab'\2}\\
&\quad - \left( S(\rho_{a\cdot b'\1 b''\1})\rho_{b'\2 b''\2} + S(\rho_{b'\1 b''\1})\lambda_a \rho_{b'\2 b''\2}\right) S(\rho_{b'\3 b''\3})\rho_{b''\4}\rho_{b'\4}\\
&\quad + S(\rho_{b'\1 b''\1})\rho_{b''\2}\rho_{b'\2} \left( S(\rho_{b'\3 a}) \rho_{b'\4} + S(\rho_{b'\3})\rho_a \rho_{b'\4}\right)\\
&\quad +  S(\rho_{b'\1 a\cdot b''\1 })\rho_{b''\2 }\rho_{b'\2} + S(\rho_{b'\1 b''\1})\rho_{b''\2}\rho_{b'\2 a} - S(\rho_{b'\1 \cdot a b''\1 })\rho_{b''\2}\rho_{b'\2} \\
&\quad - S(\rho_{b'\1 b''\1})\rho_{ab''\2}\rho_{b'\2}.
\end{align*}
Since $a$ belongs to the generalized alternative nucleus of $U(\m)$, the sum of the terms with a factor of the form $S(\rho_z)$ 
where $z$ is a product of all $a, b'$ and $b''$ vanish. 
The remaining terms have a common left factor $S(\rho_{b'\1 b''\1})$ which is not needed to check that $S_1 = 0$ and so we have to prove that
\begin{align*}
\sum \rho_{b''}\rho_{b'\1}&S(\rho_{ab'\2}) \rho_{b'\3} + \rho_{b''} \lambda_a \rho_{b'} + \rho_{b''}\rho_{ab'} - \lambda_a \rho_{b''}\rho_{b'}\\
 &+ \sum \rho_{b''\1}\rho_{b'\1}S(\rho_{b'\2 a})\rho_{b'\3} + \rho_{b''}\rho_a\rho_{b'} + \rho_{b''}\rho_{b'a}- \rho_{ab''}\rho_{b'} = 0.
\end{align*}
By Proposition \ref{prop:lambdaRecursive} we know that $\rho_{az+za} = \rho_a\rho_z + \rho_z \rho_a$, so after simplifying, 
it is equivalent to show that
\begin{displaymath}
	\rho_{b''} \lambda_a \rho_{b'} - \lambda_a \rho_{b''}\rho_{b'} + \rho_{b''}\rho_a\rho_{b'}- \rho_{ab''}\rho_{b'} = 0,
\end{displaymath}
which is an immediate consequence of Proposition \ref{prop:lambdaRecursive}.

We now prove that $S_2 = 0$. After expanding $S_2$ in terms of $\lambda_z$ and $\rho_z$ we get
\begin{align*}
S_2 &= \sum S(\rho_{ab'\1\cdot b''\1}) \lambda_{b'\2}\rho_{b''\2} + S(\rho_{b'\1 b''\1}) \lambda_{ab'\2}\rho_{b''\2} \\
&\quad - \left( S(\rho_{a\cdot b'\1 b''\1}) + S(\rho_{b'\1 b''\1})\lambda_a \right) \rho_{b'\2 b''\2}S(\rho_{b'\3 b''\3})\lambda_{b'\4}\rho_{b''\4}\\
&\quad +  S(\rho_{b'\1 a\cdot b''\1})\lambda_{b'\2}\rho_{b''\2} + S(\rho_{b'\1b''\1})\lambda_{b'\2 a}\rho_{b''\2} \\
&\quad - S(\rho_{b'\1 b''\1})\lambda_{b'\2}\rho_{b''\2} \left( S(\rho_{ab''\3})\rho_{b''\4} + S(\rho_{b''\3})\lambda_a\rho_{b''\4}\right)\\
&\quad - \left( S(\rho_{b'\1\cdot ab''\1}) \lambda_{b'\2}\rho_{b''\2} + S(\rho_{b'\1 b''\1})\lambda_{b'\2}\rho_{ab''\2}\right).
\end{align*}
Again, since $a$ belongs to the generalized alternative nucleus of $U(\m)$, the sum of the terms with a left factor of the form $S(\rho_z)$ 
where $z$ is a product of all $a, b'$ and $b''$ is zero. T
he remaining terms share a left factor $S(b'\1 b''\1)$ that can be omitted in our considerations. 
Thus, to get that $S_2 = 0$ we have to prove that
\begin{align*}
\lambda_{ab'} \rho_{b''} &- \lambda_a \lambda_{b'}\rho_{b''} + \lambda_{b'a}\rho_{b''} \\
&- \sum  \lambda_{b'}\rho_{b''\1}\left( S(\rho_{ab''\2})\rho_{b''\3} + S(\rho_{b''\2})\lambda_a\rho_{b''\3}  \right) - \lambda_{b'}\rho_{ab''}  = 0,
\end{align*}
which follows directly from Proposition \ref{prop:lambdaRecursive}.
\end{proof}

\begin{thrm}
	Let $F\colon \field[\m \times \m] \rightarrow \m$ be a formal Moufang loop. 
	The category of relative modules of $F$, the category of relative modules of $\m$ 
	and the category of modules of the Lie algebra $\Lie(\m)_+$ are equivalent.
\end{thrm}
%
%
\def\cprime{$'$}

\end{document}